\documentclass[a4paper,10pt]{article}
\usepackage[utf8x]{inputenc}
\pdfoutput=1

\usepackage{nccfoots}
\usepackage{times}
\usepackage{geometry}
\usepackage[T1]{fontenc}
\usepackage{color}
\usepackage{amsmath}
\usepackage{amssymb}
\usepackage{array}
\usepackage{amsthm}
\usepackage{graphicx}
\usepackage{hyperref}
\usepackage{setspace}
\usepackage{stmaryrd}
\usepackage{longtable}
\usepackage{tabularx}
\usepackage{multirow}
\usepackage{caption}
\usepackage{esint}

\usepackage{xifthen}
\usepackage{csquotes}

\theoremstyle{plain}
\newtheorem{thm}{Theorem}[section]
\newtheorem{prop}[thm]{Proposition}
\newtheorem{cor}[thm]{Corollary}
\newtheorem{lem}[thm]{Lemma}

\theoremstyle{definition}
\newtheorem{defn}[thm]{Definition}

\newtheorem{exmp}[thm]{Example}

\theoremstyle{remark}
\newtheorem{rem}[thm]{Remark}

\theoremstyle{plain}

\newcommand{\Z}{\mathbb{Z}}
\newcommand{\R}{\mathbb{R}}

\newcommand{\C}{\mathbb{C}}

\newcommand{\N}{\mathbb{N}}

\newcommand{\supp}{\mathrm{supp}}

\renewcommand{\O}{\mathcal O}

\newcommand{\identity}{\mathrm{id}}
\newcommand{\scal}{\mathrm{scal}}
\newcommand{\ric}{\mathrm{Ric}}
\newcommand{\trace}{\mathrm{tr}}
\newcommand{\kernel}{\mathrm{ker}}
\newcommand{\End}{\mathrm{End}}

\newcommand{\dv}{\text{ }dV}

\newcommand{\spectrum}{\mathrm{spec}}

\newcommand{\ol}{\overline}
\newcommand{\wt}{\widetilde}
\newcommand{\wh}{\widehat}

\newcommand{\ddtzero}[1][t]{\left.\frac{d}{dt}\right|_{#1=0}}
\newcommand{\SetDefine}[2]{
	\ifthenelse{\isempty{#2}}
	{\left\{#1\right\}}
	{\left\{\vphantom{#2}#1\right.\left|\,\vphantom{#1}#2\right\}}
}
\newcommand{\after}{\circ}	%
\newcommand{\todo}[1]{}
\renewcommand{\todo}[1]{\textcolor{red}{\texttt{\textbf{[#1]}}}}

\newcommand{\ResonanceDominated}[1]{\spectrum\left(#1\right)\cap \left[-\frac{(n-2)^2}{4},0\right) = \SetDefine{-\frac{(n-2)^2}{4}}{}}
\newcommand{\NotResonanceDominated}[1]{\spectrum\left(#1\right)\cap \left[-\frac{(n-2)^2}{4},0\right) \neq \SetDefine{-\frac{(n-2)^2}{4}}{}}

\renewcommand{\title}[1]{{\bfseries #1}\par}
\renewcommand{\author}[1]{\medskip{#1}\par\smallskip}
\newcommand{\affiliation}[1]{{\itshape #1}\par}
\newcommand{\email}[1]{E-mail:~\texttt{#1}\par}

\numberwithin{equation}{section}

\begin{document}
\begin{center}
	\title{\LARGE Optimal coordinates for  Ricci-flat conifolds}
	\vspace{3mm}
	\author{\Large Klaus Kr{\"{o}}ncke$^1$ and {\'{A}}ron Szab{\'{o}}$^2$}
	\vspace{3mm}
	\affiliation{$^1$ Department of Mathematics,  KTH Royal Institute of Technology
 \\Lindstedtsv\"{a}gen 25\\11428 Stockholm, Sweden \\ $^2$ Institute of Astronomy,
	 Nicolaus Copernicus University \\ Grudzi{\k{a}}dzka 5 \\ 87-100 Toru{\'{n}}, Poland}
	\email{kroncke@kth.se\\aron.szabo@v.umk.pl} 
\end{center}
\vspace{2mm}
\begin{abstract}
	We compute the indicial roots of the Lichnerowicz Laplacian on Ricci-flat cones and give a detailed description of the corresponding radially homogeneous tensor fields in its kernel. For a Ricci-flat conifold $(M,g)$ which may have asymptotically conical as well as conically singular ends, we compute at each end a lower bound for the order with which the metric converges to the tangent cone. As a special subcase of our result, we show that any Ricci-flat ALE manifold $(M^n,g)$ is of order $n$ and thereby close a small gap in a paper by Cheeger and Tian.
\end{abstract}

\section{Introduction and main results}
Ricci-flat metrics belong perhaps to the most interesting class of Riemannian metrics studied in differential geometry and theoretical physics. Compact Ricci-flat manifolds are particularly hard to find.
For a long time, any known compact Ricci-flat manifold was actually flat until the resolution of the Calabi conjecture by Yau provided the existence of other examples \cite{Yau78}.

It is easier to construct Ricci-flat metrics on noncompact manifolds where a lot of recent work has been focusing on. The Ricci-flat manifolds we are considering in this paper are allowed to have a finite number of ends which are either \emph{asymptotically conical} or \emph{conically singular}. We will refer to these with the unifying notion of \emph{conifolds}.

Various examples of asymptotically conical Ricci-flat manifolds have been discovered so far. In the subclass of Ricci-flat \emph{asymptotically locally Euclidean} (ALE for short) manifolds, many examples are provided by Kronheimer's classification \cite{Kro89}. Asymptotically conical Ricci-flat manifolds which are not ALE were found in \cite{BS89,Boh99}. Various additional examples were found over the last decade, see e.g.\ \cite{CH13,Chi19,FHN21}.
On the other hand, many examples of Ricci-flat orbifolds are known (e.g.\ noncollapsed limits of compact Ricci-flat 4-manifolds \cite{And90}) and conically singular (nonorbifold) Ricci-flat metrics were constructed in \cite{HS17}.
 
The goal of this paper lies in computing for each end of an arbitrary Ricci-flat conifold the precise order. We compute a formula for the optimal decay, given entirely in terms of spectral data on the link of the cone. This extends previous work of \cite{BKN89,CT94,CH13}, who focus on the ALE situation or the K{\"{a}}hler case.
 
A necessary step in order to compute this order lies in computing possible growth and decay rates for solutions of the linearization of the equation $\ric=0$ on the tangent cone. The linearization is up to a gauge term given by the \emph{Lichnerowicz Laplacian} and the growth and decay rates are known as \emph{indicial roots}. We give a complete computation of all indicial roots, based on the idea of commuting operators and the formulas we get are surprisingly simple. 

\subsection{Indicial roots of the Lichnerowicz Laplacian on Ricci-flat cones}
Let $\wh{M}^{n-1}$ be a closed manifold and let $\ol{P}$ be a self-adjoint Laplace type operator on $\ol{M}=(0,\infty)\times \wh{M}$, which is of the form
\begin{align}\label{eq:conical_structure_0}
	\ol{P}=-\partial_{rr}^2-\frac{n-2}{n}\partial_r+\frac{1}{r^2}\wh{P}
\end{align}
for some Laplace type operator $\wh{P}$ on $\wh{M}$. The Laplace--Beltrami operator with respect to the cone metric on $\ol{M}$ is the simplest operator of this type, but we also allow more general Laplace type operators on vector bundles. The Hodge Laplacian acting on the exterior algebra and the aforementioned Lichnerowicz Laplacian on symmetric 2-tensors are typical examples. If $\ol{P}$ is of such a form and $\nu\in\R$ is an eigenvalue of $\wh{P}$, we call the (possibly complex) values
\begin{align}\label{eq:indicial_roots}
	\xi_{\pm}(\nu):=-\frac{n-2}{2}\pm\sqrt{\frac{(n-2)^2}{4}+\nu}
\end{align}
\emph{indicial roots} of $\ol{P}$ (here, we use the convention $\sqrt{x}:=\sqrt{|x|}\cdot i$ for $x<0$). Their union is called \emph{indicial set} of $\ol{P}$. The real parts of  the indicial roots correspond to possible growth and decay rates of radially homogeneous solutions of the equation
\begin{align*}
	\ol{P}u=0.
\end{align*}
For the Laplace--Beltrami operator on the cone $\ol{M}$, the indicial roots are just calculated from the Laplace eigenvalues on $\wh{M}$. For more general operators on vector bundles which are of this form, these are in general much harder to compute.

For an operator like the Lichnerowicz Laplacians, it seems to be particularly difficult, since the Lichnerowicz Laplacians on $\wh{M}$ and $\ol{M}$ are in a very complicated relation to each other, see the formulas in \cite[Lemma~7.4]{GMS18} and \cite[Lemma~4.3]{Del07} which are lenghty and not very practical for our purposes. This comes from the fact that one has to split up the symmetric $(0,2)$-tensors on $\ol{M}$ in radial, tangential and mixed components and the covariant derivative does not preserve this splitting.
Our first main result overcomes these problems in the case of Ricci-flat cones.
\begin{thm}\label{mainthm:indicial_roots_LL}
	Let $(\ol{M}^n,\ol{g})$ be a Ricci-flat cone over a closed manifold $(\wh{M}^{n-1},g)$ with $\wh{\ric}=(n-2)\wh{g}$.\\ Let $0=\lambda_0<\lambda_1\ldots$ be the eigenvalues of the Laplace--Beltrami operator on $\wh{M}$, $\mu_1<\mu_2<\ldots$ be the eigenvalues of the connection Laplacian on divergence-free 1-forms on $\wh{M}$ and $\kappa_1<\kappa_2<\ldots$ be the eigenvalues of the Einstein operator on transverse and traceless tensors on $\wh{M}$.
	\begin{itemize}
		\item[(i)]  The indicial set of the Lichnerowicz Laplacian $\ol{\Delta}_L$ on $\ol{M}$ is given by
		\begin{align*}
			&\SetDefine{\xi_{\pm}(\kappa_i),\xi_{\pm}(\mu_i+1)- 1,\xi_{\pm}(\mu_i+1)+ 1
			,\xi_{\pm}(\lambda_i)-2,\xi_{\pm}(\lambda_i),\xi_{\pm}(\lambda_i)+2
			}{i\in\N}\\
			&\qquad\cup\SetDefine{-n,2-n,0,2}{}.
		\end{align*}
		\item[(ii)] The indicial set of $\ol{\Delta}_L$ on tensors satisfying the linearized Bianchi gauge is given by
		\begin{align*}
			\SetDefine{\xi_{\pm}(\kappa_i),\xi_{\pm}(\mu_i+1)- 1
			,\xi_{\pm}(\lambda_i)-2,\xi_{\pm}(\lambda_i)
			}{i\in\N}\cup\SetDefine{-n,0}{}.
		\end{align*}
		\item[(iii)] The indicial set of $\ol{\Delta}_L$ on tensors satisfying the linearized Bianchi gauge, but which are not Lie derivatives, is given by
		\begin{align*}
			E:=\SetDefine{\xi_{\pm}(\kappa_i),\xi_{\pm}(\lambda_i)
			}{i\in\N}.
		\end{align*}
	\end{itemize}
\end{thm}
The main result follows from Theorem~\ref{thm_spectrum_tangential_LL} and Proposition~\ref{prop:Bianchi_indicial_values} below. We find the indicial roots by writing down all possible growth and decay rates of  radially homogeneous tensors in $\kernel(\ol{\Delta}_L)$.  Let us briefly outline how we get all these tensors (recall that a tensor with vanishing trace and divergence is called a TT-tensor):
\begin{itemize}
	\item If $h$ is a TT-tensor on $\wh{M}$, any $r^{\alpha}h$ is a TT-tensor on $\ol{M}$. It is quite straightforward to show that if $\wh{\Delta}_Eh=\kappa h$ (where $\wh{\Delta}_E$ is the Einstein operator on $\wh{M}$), then $r^{\xi_{\pm}(\kappa)}h\in\kernel(\ol{\Delta}_L)$.
	\item If $v\in C^{\infty}(\wh{M})$ satisfies $\wh{\Delta} v=\lambda v$, then $r^{\xi_{\pm}(\lambda)}v$ is harmonic on $\ol{M}$ and therefore, we also have  that $r^{\xi_{\pm}(\lambda)}v\cdot \ol{g}\in \kernel(\ol{\Delta}_L)$. 
\end{itemize}
In these two cases, the TT and the conformality condition simplify the calculations to a great extent. It would be however far too complicated to work out the general formulas after relaxing these conditions. Instead, 
in order to get all other $r$-homogeneous tensors in $\kernel(\ol{\Delta}_L)$, we exploit commutation formulas involving $\ol{\Delta}_L$ and other operators to a great extent:
\begin{itemize}
	\item If $\omega$ is a divergence-free $1$-form on $\wh{M}$ such that $\wh{\Delta}_1\omega:=\wh{\nabla}^*\wh{\nabla}\omega=\mu\cdot \omega$, then a short calculation shows that $r^{\xi_{\pm}(\mu+1)}\omega$ is again divergence free and $\ol{\Delta}_1(r^{\xi_{\pm}(\mu+1)}\omega)=0$. If we now apply the symmetric part $\ol{\delta}^*$ of the covariant derivative, we get $\ol{\Delta}_L(\ol{\delta}^*(r^{\xi_{\pm}(\mu+1)}\omega))=0$ due to a commutation formula for Ricci-flat manifolds. In addition, we have $\ol{\delta}^*(r^{\xi_{\pm}(\mu+1)}\omega)=\mathcal{O}(r^{\xi_{\pm}(\mu+1)-1})$ and so, $\xi_{\pm}(\mu+1)-1$ are both indicial roots. Because $\ol{\Delta}_L$ is also of the form \eqref{eq:conical_structure_0} all indicial roots must be of the form \eqref{eq:indicial_roots}. Thus, there are further indicial roots which are dual to those two. More precisely, $\xi_{+}(\mu+1)+1$ appears as the weight dual to $\xi_{-}(\mu+1)-1$ and $\xi_{+}(\mu+1)+1$ appears as the weight dual to $\xi_{-}(\mu+1)-1$.
	\item We extend this kind of argumentation when we discuss further indicial roots coming from eigenfunctions of the Laplace--Beltrami operator. If $v\in C^{\infty}(\wh{M})$ satisfies $\wh{\Delta} v=\lambda v$, then $v_{\pm}:=r^{\xi_{\pm}(\lambda)}v$ are harmonic on $\ol{M}$. By commutation, $d v_{\pm}\in\kernel(\ol{\Delta}_1)$, with $d v_{\pm}=\mathcal{O}(r^{\xi_{\pm}(\lambda)-1})$. Thus, $\xi_{\pm}(\lambda)-1$ are both indicial roots of $\ol{\Delta}_1$ and as in the previous duality argument, there are also the dual weights $\xi_{\mp}(\lambda)+1$, with corresponding harmonic forms 
	$\omega_+=r^{\xi_{+}(\lambda)-\xi_{-}(\lambda)+2}d v_{-}$ and $\omega_-=r^{\xi_{-}(\lambda)-\xi_{+}(\lambda)+2}d v_{+}$.
	Applying $\ol{\delta}^*$ to $dv_{\pm}$ and $\omega_{\pm}$ yields four elements in $\ker(\ol{\Delta}_L)$ with decay rates
	\begin{align*}
		\SetDefine{(\xi_{\pm}(\lambda)- 1)-1,(\xi_{\pm}(\lambda)+ 1)-1}{}=\SetDefine{\xi_{\pm}(\lambda),\xi_{\pm}(\lambda)-2}{}.
	\end{align*}
	Again, duality implies that $\xi_{\pm}(\lambda)+2$ is also an indicial root as it is dual to $\xi_{\mp}(\lambda)-2$.  Finally, if $\lambda=0$ (and hence $v$ is constant), some of the constructed tensors vanish, which is why we need to deal only with the four values
	\begin{align*}
		\SetDefine{-n,2-n,0,2}{}=\SetDefine{\xi_-(0)-2,\xi_-(0),\xi_+(0),\xi_+(0)+2}{}.
	\end{align*}
\end{itemize}
To conclude the proof, one of course needs to show that the above arguments have, in fact, constructed all the indicial roots.
The method explained here not only allows us to compute the indicial roots in an efficient way, but also enables us to read of almost directly which of the corresponding tensors are geometrically essential (cf. Theorem~\ref{mainthm:indicial_roots_LL} (ii) and (iii)).

As a byproduct of our result, we get a new proof of the following theorem (cf. page~\pageref{proof_linear_stability}):
\begin{thm}[{\cite[Theorem~1.2]{Kro17}}]\label{thm_linear_stability}
	With the notation of Theorem~\ref{mainthm:indicial_roots_LL}, we have $\ol{\Delta}_L\geq0$ in the $L^2$-sense if and only if the TT-eigenvalues of the Einstein operator satisfy
	\begin{align*}
		\kappa_i\geq -\frac{(n-2)^2}{4}\qquad \text{ for all }i\in\N.
	\end{align*}
\end{thm}
The condition $\ol{\Delta}_L\geq0$ is referred in the literature as the condition of \emph{linear stability}, which appears in the study of the Einstein--Hilbert action (see e.g.\ \cite[Chapter 4G]{Bes08}) and in the study of dynamical stability of fixed points of the Ricci flow (see e.g.\ \cite{Ses06}).

The proof in \cite{Kro17} is based on a suitable decomposition of the space of symmetric $(0,2)$-tensors on $\ol{M}$ and many tedious $L^2$-estimates. The new proof presented in this paper is much more stringent and due to Theorem~\ref{mainthm:indicial_roots_LL}, the structure of $\ol{\Delta}_L$ on the cone is now understood in much greater detail.

Besides its importance for the computation of the order of Ricci-flat AC/CS ends (cf. Definition~\ref{defn_conifolds}), we think that Theorem~\ref{mainthm:indicial_roots_LL} is also of great independent interest. Moreover, there are further potential applications, for example the desingularization of Einstein conifolds by smooth Einstein metrics. In addition, the method of using commuting operators to compute indicial roots can also be used for other operators, in particular the Hodge Laplacian for the exterior algebra, for which the exterior derivative and its adjoint would serve as commuting operators. This could in turn be used to compute convergence rates of other geometric structures at infinity. 

\subsection{The order of Ricci-flat conifolds}\label{subsec:order}
Let us start this subsection by making the notion of conifolds precise.
\begin{defn}\label{defn_conifolds}\label{def:manifold-with-ends}
	A manifold $M^n$ is called a manifold with ends if there exists a compact subset $K\subset M$, called the core of $M$, such that 
	\begin{itemize}
		\item[(i)] $M\setminus K$ consists of a disjoint union of connected manifolds $M_i$, $i=1,\ldots N$, called the ends of $M$, and 
		\item[(ii)] for each $i=1,\ldots N$, there exists a closed manifold $\wh{M}_i$ such that $M_i$ is diffeomorphic to the manifold $(0,\infty)\times \wh{M}_i$. 
	\end{itemize}
\end{defn}
\noindent In the following, we use the convention $\nabla^0:=\identity$.
\begin{defn}\label{def:conifold}
	Let $M^n$ be a manifold with ends $M_i$, $i=1,\ldots, N$, endowed with a Riemannian metric~$g$.
	\begin{itemize}
		\item[(i)] We say that an end $M_i$ is called \emph{asymptotically conical} if there exist constants $\tau_i,R_i>0$, a Riemannian metric $\wh{g}_i$ on $\wh{M}_i$ and a diffeomorphism $\varphi_i: M_i\to (R_i,\infty)\times \wh{M}_i$ such that
		\begin{align*}
			|\ol{\nabla}^k((\varphi_i)_*g-\ol{g_i})|_{\ol{g_i}}=\mathcal{O}(r^{-\tau_i-k}),\qquad r \to\infty,
		\end{align*}
		where $\ol{g}_i=dr^2+r^2\wh{g}_i$ denotes the cone metric on $(0,\infty)\times \wh{M}_i$ and $\ol{\nabla}$ denotes its Levi-Civita connection.
		\item[(ii)] We say that an end $M_i$ is called \emph{conically singular} if there exist constants $\tau_i,R_i>0$, a Riemannian metric $\wh{g}_i$ on $\wh{M}_i$ and a diffeomorphism $\varphi_i: M_i\to (0,R_i)\times \wh{M}_i$ such that
		\begin{align*}
			|\ol{\nabla}^k((\varphi_i)_*g-\ol{g_i})|_{\ol{g_i}}=\mathcal{O}(r^{+\tau_i-k}),\qquad r \to 0,
		\end{align*}
		where, again, $\ol{g}_i=dr^2+r^2\wh{g}_i$ denotes the cone metric on $(0,\infty)\times \wh{M}_i$ and $\ol{\nabla}$ denotes its Levi-Civita connection.
	\end{itemize}
	In both cases, $\varphi_i$ is called an asymptotic chart and $\tau_i$ is called an order of the end $M_i$.
	We call the manifold $((0,\infty)\times \wh{M}_i,\ol{g}_i)$ the \emph{tangent cone} of the end $M_i$, and we say that $M_i$ is tangent to its tangent cones.
	Furthermore, we call $(M,g)$
	\begin{itemize}
		\item[(iii)] \emph{asymptotically conical} (AC for short) if all ends are asymptotically conical;
		\item[(iv)] \emph{conically singular} (CS for short) if all ends are conically singular and
		\item[(v)] \emph{conifold} (AC/CS for short) if each end is either asymptotically conical or conically singular.
	\end{itemize}
\end{defn}
\noindent Evidently, the order of an end is not unique: if an end is of order $\tau$, then it is also of order $\tau'$ for any $\tau'<\tau$.
Our goal in this paper is to find a lower bound for the order which is as large as possible.

\noindent To cover a special situation, which we later call the \emph{resonance-dominated} case, we introduce a slightly refined notion of asymptotic order:
\begin{defn}\label{def:weak-order-of-a-conifold}
	With the notation of Definition~\ref{defn_conifolds},
	an asymptotically conical end $(M_i,g_i)$ of a conifold is called asymptotically conical \emph{weakly of order $\tau_i$} if there exists an asymptotic chart $\phi_i$ in which we have for all $k\in\N$ that
	\begin{align*}
		|\ol{\nabla}^k(\varphi_i)_*g-\ol{g_i}|_{\ol{g_i}}=\mathcal{O}(r^{-\tau_i-k}\log(r)),\qquad r \to\infty.
	\end{align*}
\end{defn}
\noindent Note that if and AC end is weakly of order $\tau$, then it is of order $\tau-\epsilon$ for all $\epsilon>0$.

\begin{center}
	\begin{figure}[tbh!]
		\centering
		\includegraphics[width=0.7\linewidth]{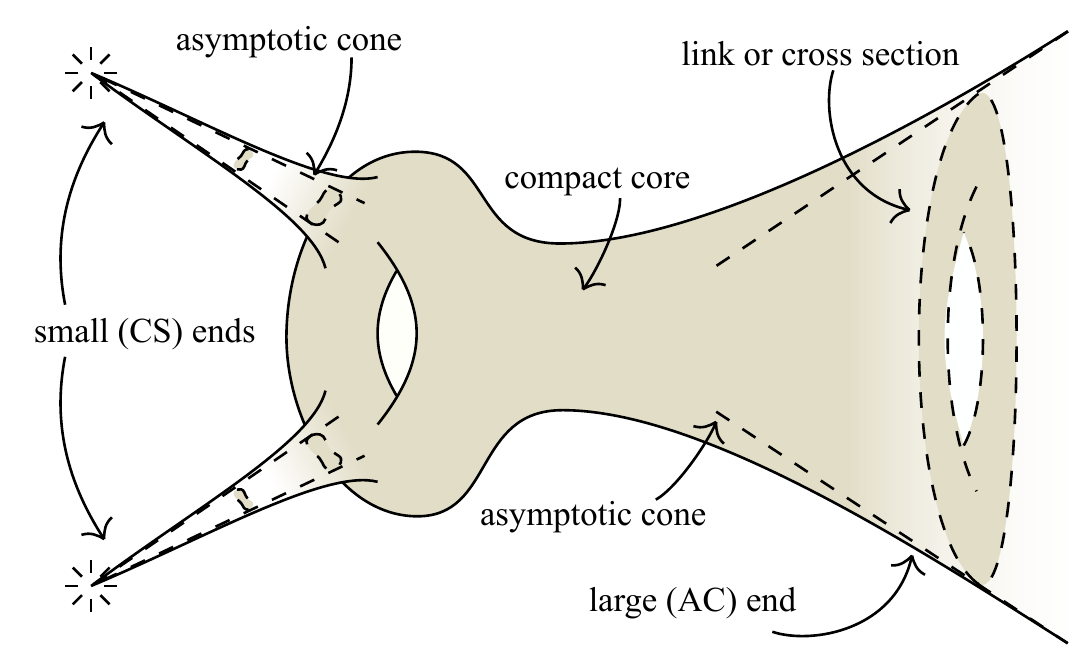}
		\caption[Schematic picture of a conifold]{Schematic picture of a conifold, cf. Definitions~\ref{def:manifold-with-ends} and \ref{def:conifold}.}
		\label{fig:conifold}
	\end{figure}
\end{center}

Now let $(\ol{M},\ol{g})$ be a Ricci-flat cone and $\lambda_i,\mu_i,\kappa_i$ and
$E$ as in Theorem~\ref{mainthm:indicial_roots_LL}. Denote by $\mathrm{Re}(E)\subset \R$ the set of real parts of elements in $E$ and define
\begin{align*}
	E_+:&=\mathrm{Re}(E)\cap(0,\infty)=\SetDefine{\xi_{+}(\kappa_i),\xi_{+}(\lambda_i)
	}{i\in\N,\kappa_i>0 }.
\end{align*}
and
\begin{align*}
	E_-:&=\mathrm{Re}(-E)\cap(0,\infty)\\
	&=\SetDefine{-\xi_{-}(\kappa_i),-\xi_{-}(\lambda_i)}{i\in\N} \cup \SetDefine{-\xi_{+}(\kappa_j)}{i\in\N,-\frac{(n-2)^2}{4}\leq\kappa_j<0} \\
	&\qquad\cup 
	\SetDefine{-\mathrm{Re}(\xi_{\pm}(\kappa_j))=\frac{n-2}{2}}{i\in\N, \kappa_j<-\frac{(n-2)^2}{4}},
\end{align*}
see also Figure~\ref{fig:generic-decay} on page~\pageref{fig:generic-decay}.
Furthermore, we define
\begin{align}\label{eq:rates}
	\xi_+:=\min E_+\qquad 	\xi_-:=\min E_-.
\end{align}
Note that both numbers are positive and depend entirely on spectral data on the link $(\wh{M},\wh{g})$ of the cone. 
Note also that there are no a~priori positive lower or upper bounds for $\xi_+$ and $\xi_-$.
We write $\xi_+(\wh{M},\wh{g})$, resp.\ $\xi_-(\wh{M},\wh{g})$ if we wish to emphasize the dependence of these values on $(\wh{M},\wh{g})$.
\begin{defn}
	We call a Ricci-flat cone \emph{resonance-dominated} if \[\SetDefine{\kappa_i}{i\in\N}\cap\left[-\frac{(n-2)^2}{4},0\right)=\SetDefine{-\frac{(n-2)^2}{4}}{}.\] An end of a conifold is called resonance-dominated if its tangent cone is resonance-dominated.
\end{defn}
\noindent Now we are able to formulate the second main result of this paper.
\begin{thm}\label{mainthm:conifold_rate}
	Let $(M^n,g)$ be a Ricci-flat conifold with ends $M_i$, $i=1,\ldots,N$, which are modeled by Ricci-flat cones over Einstein manifolds $(\wh{M}_i,\wh{g}_i)$. Then the following assertions hold:
	\begin{itemize}
		\item[(i)] If $M_i$, $i\in\SetDefine{1,\ldots,N}{}$, is an asymptotically conical end, then it is of order $\xi_-(\wh{M}_i,\wh{g}_i)$ if it is not resonance-dominated and weakly of order $\frac{n-2}{2}$ otherwise.
		\item[(ii)] If $M_i$, $i\in\SetDefine{1,\ldots,N}{}$, is a conically singular end, then it is of order $\xi_+(\wh{M}_i,\wh{g}_i)$.
	\end{itemize} 
\end{thm}
\begin{rem}\label{rem:decay}
	We do not claim here that the largest possible order of each end is exactly $\xi_{\pm}(\wh{M}_i,\wh{g}_i)$. For some conifolds, the order could be a~priori larger. However one would need to study this with methods adjusted to these specific situations and can not be treated in this general framework.
	In this paper, we will use the adjective \enquote{optimal} in this sense.
\end{rem}
The construction of the asymptotic chart is based on a global slice theorem for the action of the diffeomorphism group on the space of metrics. We impose the \emph{Bianchi gauge},
which is made precise in Definition~\ref{def:Bianchi_gauge} below.

We will prove that given a conifold metric $g$ on a manifold $M$, there exists for any metric $\tilde{g}$ sufficiently close to $g$ in a suitable topology (with respect to weighted Sobolev spaces) a diffeomorphism $\psi$ on $M$ such that up to an arbitrarily small open subset $U$ (which we are free to choose), $g$ is in Bianchi gauge with respect to $\psi^*\tilde{g}$.
Given a Ricci-flat conifold $(M,g)$, we may choose now $\tilde{g}$ such that it agrees with $g$ in the core of $M$ and (with respect to given asymptotic charts $\varphi_i$) agrees with the exact cone metrics $\ol{g}_i$ at the ends. With the help of Theorem~\ref{mainthm:indicial_roots_LL}, we then compute the order of $g-\psi^*\tilde{g}$.  As a consequence, the new asymptotic charts $\varphi_i\circ\psi$ will give us the desired orders.

In contrast to \cite{BKN89,CT94} who construct the gauge locally at the (single) end, our gauge is constructed \emph{globally} on all of the manifold and gives us the \enquote{optimal} asymptotic charts for all ends at once.
\begin{rem}\label{rem:small_rates}
	Intuitively, one may think of an asymptotic order $\tau$ at infinity to be \enquote{small} if $\tau<n-2$, since $n-2$ is the decay rate of the fundamental solution of the Laplace equation. By Theorem~\ref{mainthm:conifold_rate}, an  AC end has small order if $E_-\cap(2-n,0)\neq\emptyset$. By definition of $E_-$,
	this is the case if some $\kappa_i$ are negative (this means that the link of the cone is \emph{unstable} with respect to the Einstein--Hilbert action, cf.\ \cite[Definition~4.63]{Bes08}).
\end{rem}
The notion of ADM mass of an asymptotically Euclidean manifold can be generalized to asymptotically conical manifolds.
\begin{defn}\label{def:ADM-mass}
	The ADM mass of an asymptotically conical manifold $(M,g)$ with a single end tangent to $((0,\infty)\times \wh{M},dr^2+r^2 \wh{g})$ is
	\begin{align*}
		m_{ADM}(M,g)=-\frac{1}{\mathrm{vol}(S^{n-1})}\lim_{R\to\infty}\int_{\SetDefine{R}{}\times \wh{M}}\langle \delta_{\ol{g}}(\varphi_*g)+d(\trace_{\ol{g}}\varphi_*g),\partial_r\rangle\dv_{R^2 \wh{g}},
	\end{align*}
	where $\varphi$ is an asymptotic chart.
\end{defn}
\noindent In view of Remark~\ref{rem:small_rates}, we obtain the following result:
\begin{cor}
	Let $(\ol{M},\ol{g})$ be a Ricci-flat cone over an Einstein manifold which is linearly stable with respect to the Einstein--Hilbert action. Then any Ricci-flat AC manifold with tangent cone $(\ol{M},\ol{g})$ has vanishing ADM mass.
\end{cor}
If all $\kappa_i$ are positive, then $\xi_->n-2$ and the mass clearly vanishes. If $\kappa_1=0$, then $\xi_-=n-2$, but the leading term of the asymptotic expansion of $\varphi_*g-\ol{g}$ is a TT-tensor, see the proof of Theorem~\ref{thm_spectrum_tangential_LL}. Therefore, the mass vanishes in this case, too.

\begin{exmp}
	In \cite{CH13}, it is shown that the AC Stenzel metric on $T^*S^m$ has order $2\frac{m}{m-1}$ and this is the optimal order. This shows that $2\frac{m}{m-1}\in E_-$. 
	Therefore, $\frac{2m}{m-1}=-\xi_+(\kappa_{i_0})$ for some $i_0\in \N$ and we get $\kappa_{i_0}<0$ if $m>2$. We deduce that for $m>2$, the link of the tangent cone of the Stenzel metric is unstable with respect to the Einstein--Hilbert action.
\end{exmp}	
\begin{exmp}
	Any 9-dimensional product Einstein manifold $\wh{M}_1\times \wh{M}_2$ with positive Ricci curvature satisfies $\kappa_1=-2(m-2)=-\frac{(m-2)^2}{4}$ (with $m=10$), see e.g.\ \cite[Section~4]{Kro15b}. Moreover, if the products are Einstein--Hilbert stable (for example, if both $\wh{M}_1$ and $\wh{M}_2$ are spheres), all other $\kappa_i$ are nonnegative. The 10-dimensional cones over such products are resonance-dominated. Ricci-flat AC manifolds tangent to such cones were constructed in \cite{Boh99} and these examples are by definition resonance-dominated as well.
\end{exmp}

In more specific geometric situations, where the tangent cones are quotients of Euclidean space, we derive the following from Theorem~\ref{mainthm:conifold_rate}
\begin{thm}\label{mainthm:orbifold_Rate}
	Every Ricci-flat orbifold is of order $2$.
\end{thm}
\begin{thm}\label{mainthm:ALE_Rate}
	Every $n$-dimensional Ricci-flat ALE manifold is ALE of order $n$.
\end{thm}
In \cite[Theorem~1.5]{BKN89}, it was shown that every ALE manifold is of order $n-1$ and it is of order $n$ if the manifold is K\"{a}hler or if $n=4$. In \cite[Theorem~5.103]{CT94}, it was already claimed that the assertion of Theorem~\ref{mainthm:ALE_Rate} holds. However, the proof contains a tiny gap which
seems to have been overlooked so far in the literature.

Let us explain this gap for completeness.
The central argument in \cite[p.\ 538]{CT94} is that any harmonic function $f$ on $\R^n/\Gamma$
with $f=\mathcal{O}(r^{1-n})$ satisfies $f=\mathcal{O}(r^{-n})$ if $\Gamma\neq\SetDefine{1}{}$. This is because the leading order term is exactly of the form $x\mapsto \langle x,a\rangle r^{-1}$ for some $a\in \R^n$ and this is not invariant under any nontrivial $\Gamma$.
For this reason, it was argued that the same decay rate should also hold for a harmonic TT-tensor. However, if $h$ is a harmonic tensor on $\R^n/\Gamma$, expanded in flat coordinates as
\begin{align*}
	h=h_{ij}dx^i\otimes dx^j,
\end{align*}
then the functions $h_{ij}$ are not necessarily $\Gamma$-invariant, so the above argument for functions cannot be applied. Consider the following example:  Equip $\R^4\cong\C^2$ with complex coordinates $(z_1,\bar{z}_1,z_2,\bar{z}_2)$, where $z_i=x_i+iy_i$ and $\ol{z}_i=x_i-iy_i$, $i=1,2$. Recall that
\begin{align*}
	\Delta=-4(\partial_{z_1}\circ\partial_{\bar{z}_1}+\partial_{z_2}\circ\partial_{\bar{z}_2}),
\end{align*}
where $\partial_{z_i}=1/2(\partial_{x_i}-i\partial_{y_i})$ and $\partial_{\bar{z}_i}=1/2(\partial_{x_i}+i\partial_{y_i})$.
The function $f:z\mapsto (z_1)^3$ is harmonic and invariant under the multiplicative action of the group $\Z_3=\SetDefine{1,e^{\frac{2i\pi}{3}},e^{\frac{4i\pi}{3}}}{}$ and so is its real part, which is given by the function $g=\mathrm{Re}(f):(x_1,y_1,x_2,y_2)\mapsto(x_1)^3-4(x_1)(y_1)^2$.
The tensor $\nabla^2g$ is now harmonic and $\Gamma$-invariant and its component functions are linear. Therefore, $h=r^{-4}\nabla^2g$ is another harmonic $\Gamma$-invariant tensor which is of order $\mathcal{O}(r^{-3})$.

On the other hand, $h=r^{-4}\nabla^2g$ is not a TT-tensor and even its trace-free part is not TT. Thus it does not contradict our results which imply that any harmonic $\Gamma$-invariant TT-tensor decaying at infinity decays of order $n$. But as the example of the function $g$ shows, this cannot be directly concluded from the fact that linear functions are not $\Gamma$-invariant for any nontrivial $\Gamma$. 

\subsection{Structure of the paper}
In Section~\ref{sec:notation}, we introduce some notations, conventions and commutation formulas which we use throughout the paper. In Section~\ref{sec:tangential_operators}, Laplace type operators on cones are discussed in great detail. After collecting a few general statements, we first discuss the connection Laplacian on 1-forms. Building upon this, we are able to unravel the complicated structure of the Lichnerowicz Laplacian on cones and prove Theorems~\ref{mainthm:indicial_roots_LL} and~\ref{thm_linear_stability}. 
Section~\ref{sec:coordinates} is devoted to the construction of optimal coordinates on conifolds. After an introduction into weighted function spaces and a discussion of decay of Ricci-flat metrics in Bianchi gauge on cones, we prove a slice theorem for the Bianchi gauge on conifolds. These decay results and the slice theorems are then used to prove Theorem~\ref{mainthm:conifold_rate} before we conclude with an application in the orbifold and ALE cases to prove Theorem~\ref{mainthm:orbifold_Rate} and Theorem~\ref{mainthm:ALE_Rate}.

\subsection*{Acknowledgments}
The authors want to thank Daniel Grieser, Uwe Semmelmann and Boris Vertman for helpful discussions. The work of the first author is supported by the DFG through the priority program 2026 \emph{Geometry at Infinity}. The second author was partially funded by the DFG through the Research Training Group 1670 \emph{Mathematics Inspired by String Theory and Quantum Field Theory}, and the National Science Center (NCN), Poland under grant number OPUS 2021/41/B/ST9/00757.
A part of this paper is based on the PhD thesis of the second author.

\section{Notations, conventions and formulas}\label{sec:notation}
For a Riemannian manifold $(M,g)$ we define the 
Riemannian curvature tensor with the sign convention such that
\begin{align}
	R_{X,Y}Z=\nabla^2_{X,Y}Z-\nabla^2_{Y,X}Z,\qquad \text{ for all }X,Y,Z\in C^{\infty}(TM),
\end{align}
where $\nabla$ is the Levi-Civita connection of $g$.
The Ricci tensor is denoted by $\ric$ and the scalar curvature by $\scal$. The Laplace--Beltrami operator $\Delta:C^{\infty}(M)\to C^{\infty}(M)$ with the sign convention such that $\Delta f=-\trace{\nabla^2}f=-g^{ij}\nabla^2_{ij}f$. With the same sign convention, we define the connection Laplacians
\begin{align}
	\Delta_1=-\trace{\nabla^2}:C^{\infty}(T^*M)\to C^{\infty}(T^*M),\qquad
	\Delta_2=-\trace{\nabla^2}:C^{\infty}(S^2M)\to C^{\infty}(S^2M),
\end{align}
on 1-forms and symmetric $(0,2)$-tensors, respectively. To distinguish between these operators and the Laplace--Beltrami operator, we sometimes also use the notation $\Delta_0$ instead of $\Delta$.
 
The symmetric tensor product of $\omega,\eta\in C^{\infty}(T^*M)$ is defined as $\omega\odot\eta=\omega\otimes\eta+\eta\otimes\omega\in C^{\infty}(S^2M)$.
The divergences of $\omega\in C^{\infty}(T^*M)$ and $h\in C^{\infty}(S^2M)$ are defined with the sign convention such that
\begin{align}
	\delta\omega=-g^{ij}\nabla_i\omega_j\in C^{\infty}(M),\qquad \delta h_k=-g^{ij}\nabla_ih_{jk}\in C^{\infty}(T^*M),
\end{align}
respectively. The formal adjoint $\delta^*:C^{\infty}(T^*M)\to C^{\infty}(S^2M)$ of $\delta$ is given by
\begin{align}
	(\delta^*\omega)_{ij}=\frac{1}{2}(\nabla_i\omega_j+\nabla_j\omega_i).
\end{align}
Note that $\delta^*d f=\nabla^2f$ for any $f\in C^{\infty}(M)$ and that $\delta^*$ is related to the Lie derivative by $2\delta^*\omega=\mathcal{L}_{\omega^{\sharp}}g$, where $\omega^{\sharp}\in C^{\infty}(TM)$ is the dual vector field of $\omega$ with respect to $g$. 
Note moreover that $\trace\after \delta^* = -\delta$.
We also have trace-free versions of $\delta^*$ and $\nabla^2$, given by
\begin{align}
	\mathring{\delta}^*\omega&:=\delta^*\omega-\frac{1}{n}\trace(\delta^*\omega)g=\delta^*\omega+\frac{1}{n}\delta\omega\cdot g,\label{eq:trace-tree-codivergence}\\
	\mathring{\nabla}^2f&:=\mathring{\delta}^*df=\nabla^2f+\frac{1}{n}\Delta f\cdot g. \label{eq:trace-free-Hessian}
\end{align}
Furthermore, we introduce the Bianchi operator, given by
\begin{equation}\label{eq:Bianchi-operator}
	B:=\delta+\frac{1}{2}d\circ \trace: C^{\infty}(S^2M)\to C^{\infty}(T^*M).
\end{equation}
This notion comes from the fact that $B$ arises from linearizing the Bianchi gauge condition $V({g},\tilde{g})=0$, where the vector field $V({g},\tilde{g})$ depending on two Riemannian metrics $g,\tilde{g}$ is given in local coordinates by
\begin{align*}
	V({g},\tilde{g})^l:=g^{ij}(\Gamma(g)_{ij}^l-\Gamma(\tilde{g})_{ij}^l).
\end{align*}
More precisely, a short calculation shows
\begin{align}\label{eq:linearized_bianchi}
	B(h)=\ddtzero V(g+th,g)=-\ddtzero V(g,g+th).
\end{align}
The Lichnerowicz Laplacian $\Delta_L:C^{\infty}(S^2M)\to C^{\infty}(S^2M)$ is defined by 
\begin{align*}
	\Delta_Lh=\Delta_2h-\ric\circ h+h\circ \ric-2\mathring{R}h,
\end{align*}
where $(h\circ k)_{ij}=g^{mn}k_{im}h_{jm}$, and the full curvature term is 
$\mathring{R}h_{ij}=g^{km}g^{ln}R_{iklj}h_{mn}$, where $R_{iklj}=g(R_{\partial_i,\partial_k}\partial_l,\partial_j)$. 
The term $\mathring{R}\in\End(S^2M)$ is called the curvature potential. 
Up to a gauge term, $1/2\cdot\Delta_L$ is the linearization of the Ricci tensor. More precisely,
we have
\begin{align}\label{eq:linearized_Ricci}
	\ddtzero\ric_{g+th}=\frac{1}{2}\Delta_Lh-\delta^*(Bh)=\frac{1}{2}(\Delta_Lh-\mathcal{L}_{(Bh)^{\sharp}}g),
\end{align}
see \cite[Theorem~1.174]{Bes08}. 
The Einstein operator $\Delta_E:C^{\infty}(S^2M)\to C^{\infty}(S^2M)$ is defined by 
\begin{align}\label{eq:Einstein-operator}
	\Delta_Eh=\Delta_2h-2\mathring{R}h.
\end{align}
Note that the two operators are related by $\Delta_L=\Delta_E+2\lambda$ whenever $\ric=\lambda\cdot g$. In this case, we additionally have
\begin{align*}
	\ddtzero \left[\ric_{g+th}-\lambda(g+th)\right]=\frac{1}{2}\Delta_Eh-\delta^*(Bh)=\frac{1}{2}(\Delta_Eh-\mathcal{L}_{(Bh)^{\sharp}}g),
\end{align*}
which is easily seen from \eqref{eq:linearized_Ricci}.
If $(M^n,g)$ is Einstein with $\ric_g=\lambda\cdot g$, we furthermore have a variety of commutation identities involving these differential operators: For $v\in C^{\infty}(M)$, $\omega\in C^{\infty}(T^*M)$ and $h\in C^{\infty}(S^2M)$, we have
\begin{align}\label{commutation}
	\Delta_1(dv)&=d(\Delta_0v-\lambda v),& \Delta_0(\delta \omega)&=\delta(\Delta_1\omega+\lambda\omega),\nonumber \\
	\Delta_E(\delta^*\omega)&=\delta^*(\Delta_1\omega-\lambda \omega),& \Delta_1(\delta h)&=\delta(\Delta_E h+\lambda h),\nonumber \\
	\Delta_E(\nabla^2v)&=\nabla^2(\Delta_0v-2\lambda v),& \Delta_0(\delta \delta h)&=\delta\delta(\Delta_Eh+2\lambda h),\nonumber \\
	\Delta_E(v\cdot g)&=(\Delta v-2\lambda v)g,& \Delta(\trace h)&=\trace(\Delta_Eh+2\lambda h).
\end{align}
The computations can be found in \cite{Lic61}, see also \cite[p.\ 8]{Kro15b}. Note that the third line follows trivially from the first and the second line and that the formulas on the right-hand side follow from the ones on the left-hand side by taking the formal adjoints. Also, note that from \eqref{commutation}, we can deduce in an obvious way other commutation formulas involving the operators $\Delta_L$, 	$\mathring{\delta}^*$, $\mathring{\nabla}^2$ and $B$. 

The spectrum of a differential operator $P$ is denoted by $\spectrum(P)$ and its strictly positive part is denoted by $\spectrum_+(P)=\spectrum(P)\cap(0,\infty)$. If $\lambda$ is an eigenvalue of the operator $P$, we denote the corresponding eigenspace by $E(P,\lambda)$. 

For a section $u$ of a Riemannian vector bundle with metric connection, we write $u=\mathcal{O}_k(r^{\alpha})$ if 
$|\nabla^lu|=\mathcal{O}(r^{\alpha-l})$ for all $l=0,\ldots k$
and $u=\mathcal{O}_{\infty}(r^{\alpha})$ if $u=\mathcal{O}_{k}(r^{\alpha})$ for all $k\in\N$.
Furthermore, we write $u=\mathcal{O}_k(r^{\alpha}\log r)$ if $|\nabla^lu|=\mathcal{O}(r^{\alpha-l}\log(r))$
for all $l=0,\ldots k$
and $u=\mathcal{O}_{\infty}(r^{\alpha}\log(r))$ if $u=\mathcal{O}_{k}(r^{\alpha}\log(r))$ for all $k\in\N$.

For a vector bundle $V$, we denote by $C^{\infty}(V)$, $C^{2,\alpha}(V)$ etc. the space of sections with respective regularity. If $P$ is a differential operator acting on sections of $V$, we denote by $\kernel_{C^{\infty}}(P)$, $\kernel_{C^{2,\alpha}}(P)$ etc. the elements in the kernel of $P$ with respective regularity, cf. \cite{Pac13}.
While $S^2M$ denotes the bundle of all symmetric (0,2)-tensors, the subset $S^2_+M\subset S^2M$ is the set of all  positive definite scalar products over all points of $M$. Although $S^2_+M$ does not form a vector bundle, we denote for notational convenience by $C^{\infty}(S^2_+M)$, $C^{2,\alpha}(S^2_+M)$ etc. the set of Riemannian metrics with the respective regularity.

There are several metrics of interest in this paper, and the interplay between them is crucial. 
For brevity, a system of diacritical marks has been implemented in the notation. 
Metrics, covariant derivatives, curvatures, Laplace type operators and bundles connected to the link, a Riemannian cone and a generic metric are denoted by hats ($\wh{g}$, $\wh{\nabla}$ etc.), overlines ($\ol{g}$, $\ol{\nabla}$ etc.) and tildes ($\tilde{g}$, $\tilde{\nabla}$ etc.), respectively. 
Moreover, geometric objects connected to the conifold metric under consideration carry no diacritical marks.
This convention is not applied to tangential operators.

\section{Ricci-flat cones and their tangential operators}\label{sec:tangential_operators}

Throughout this section, we fix a smooth closed Riemannian manifold $(\wh{M},\wh{g})$ of dimension $n-1$ (with $n>2$) and its $n$-dimensional cone, denoted by
\begin{align}
	(\ol{M},\ol{g})=(\R_+\times \wh{M},dr^2+r^2\wh{g}),
\end{align}
where $r$ is the canonical coordinate on $\R$.
We will use the canonical projections to pull back objects on $\wh{M}$ to objects on $\ol{M}$. For notational convenience, we will drop the explicit reference to the projections.
We denote the indices corresponding to coordinates on $\wh{M}$ by $i,j,k,\ldots$, and the index $r$ refers to the $r$-coordinate in the manifold $\ol{M}$.
Let us denote the indices corresponding to coordinates on $\ol{M}$ by $\alpha,\beta,\gamma$.
The Christoffel symbols on  $\ol{M}$ are related to the ones on $\wh{M}$ by
\begin{align}\label{christoffelA2}
	\ol{\Gamma}_{ij}^k&=\wh{\Gamma}_{ij}^k,\qquad \ol{\Gamma}_{ij}^{r}=-r\cdot g_{ij},\qquad
	\ol{\Gamma}_{ir}^j=\ol{\Gamma}_{ri}^j=\frac{1}{r}\delta_i^j,
\end{align}
while the other Christoffel symbols vanish.
Therefore. the curvature tensors of $\ol{g}$ and $\wh{g}$ are related by
\begin{align}
	\ol{R}_{ijkl}=r^2(\wh{R}_{ijkl}+(\wh{g}_{ik}\wh{g}_{jl}-\wh{g}_{il}\wh{g}_{jk})),
\end{align}
while the other components of $\ol{R}$ vanish. Consequently, the Ricci tensors are related by
\begin{align*}
	\ol{\ric}_{ij}=\wh{\ric}_{ij}-(n-2)\wh{g}_{ij},
\end{align*}
while the other components of $	\ol{\ric}_{ij}$ vanish. In particular, $(\ol{M},\ol{g})$ is Ricci-flat if and only if $(\wh{M},\wh{g})$ is Einstein with $\wh{\ric}=(n-2)\wh{g}$. 
Note that this latter condition can be achieved via a homothetic rescaling for
any positive Einstein manifold. 
In this situation, we call $(\ol{M},\ol{g})$ the Ricci-flat cone over $(\wh{M},\wh{g})$. We will assume from now on that $(\ol{M},\ol{g})$ is Ricci-flat.

\subsection{Conical and tangential operators}
Let $\ol{V}$ be a Riemannian vector bundle over $\ol{M}$ with metric connection $\ol{\nabla}$ and denote its restriction by $\wh{V}:=\ol{V}_{\SetDefine{1}{}\times \wh{M}}$. By identifying the bundle restrictions $\wh{V}|_{\SetDefine{r}{}\times \wh{M}}$ with each other via parallel transport along radial lines (i.e. curves of the form $\gamma(t):=(t,p)$ for fixed $p\in \wh{M}$), we get a natural identification
\begin{align}\label{eq:bundle_identification}
	C^{\infty}(\ol{V})\cong C^{\infty}(\R_{+},C^{\infty}(\wh{V})).
\end{align}
Via this parallel transport, we will sometimes consider sections in $\wh{V}$ as sections in $\ol{V}$, without mentioning the extension of the domain from $\wh{M}\cong \SetDefine{1}{}\times \wh{M}$ to $\ol{M}$ explicitly.
Let us assume that all the induced connections ${}^{r}\nabla$ on $\ol{V}|_{\SetDefine{r}{}\times \wh{M}}\cong \wh{V}$ do coincide. We denote this connection on $\wh{V}$ by $\wh{\ol{\nabla}}$.
\begin{defn}\label{def:conical-operator}
	A self-adjoint Laplace type operator $\ol{\Delta}_{\ol{V}}$ acting on $C^{\infty}(\ol{V})$ is called a \emph{conical operator} if it is with respect to \eqref{eq:bundle_identification} of the form
	\begin{align}\label{eq:conical_structure}
		\ol{\Delta}_{\ol{V}}=-\partial^2_{rr}-\frac{n-1}{r}\partial_r+\frac{1}{r^2}\Box_{\wh{V}},
	\end{align}
	where ${\Box}_{\wh{V}}$ is a self-adjoint Laplace type operator acting on $C^{\infty}({\wh{V}})$. We we call ${\Box}_{\wh{V}}$ the \emph{tangential operator} of~$\ol{\Delta}_{\ol{V}}$.
\end{defn}
\begin{lem}\label{lem:conical_connections}
	If $\ol{\Delta}_{\ol{V}}=\ol{\nabla}^*\ol{\nabla}$ is the connection Laplacian of $\ol{V}$, then it is a conical operator and its tangential operator is the connection Laplacian $\wh{\ol{\nabla}}^*\wh{\ol{\nabla}}=-\wh{g}^{ij}\left(\ol{\nabla}_{i}\ol{\nabla}_{j}-\ol{\nabla}_{\wh{\nabla}_{\partial_{i}}\partial_{j}} \right)$ of the induced connection $\wh{\ol{\nabla}}$ on $\wh{V}$ with respect to the metric $\wh{g}$.
\end{lem}
\begin{proof}
	This simply follows from writing
	\begin{align*}
		\ol{\nabla}^*\ol{\nabla}=-\ol{g}^{{\alpha\beta}}(\ol{\nabla}_{\alpha}\ol{\nabla}_{\beta}+\ol{\nabla}_{\ol{\nabla}_{\partial_{\alpha}}\partial_{\beta}})
		=-\ol{\nabla}_{r}\ol{\nabla}_{r}
		-\frac{n-1}{r}\ol{\nabla}_{r}
		-r^{-2}\wh{g}^{{ij}}\left(\ol{\nabla}_{i}\ol{\nabla}_{j}-\ol{\nabla}_{\wh{\nabla}_{\partial_{i}}\partial_{j}} \right),
	\end{align*}
	and observing that $\ol{\nabla}_{r}\sim\partial_r$ via \eqref{eq:bundle_identification}.
\end{proof}
\begin{lem}\label{lem:conical_Schroedinger_operators}
	If $\ol{\Delta}_{\ol{V}}=\ol{\nabla}^*\ol{\nabla}+\ol{A}$ with a potential $\ol{A}\in C^{\infty}(\End(\ol{V}))$, then $\ol{\Delta}_{\ol{V}}$ is a conical operator if and only if $\ol{R}=r^{-2}\wh{A}$ for some potential $\wh{A}\in C^{\infty}(\End(\wh{V}))$.
\end{lem}
\begin{proof}
	This is obvious.
\end{proof}
\begin{exmp}
	The Laplace--Beltrami operator on $\ol{M}$ is a conical operator and its tangential operator is the Laplace--Beltrami operator on $\wh{M}$.
\end{exmp}
\begin{exmp}
	Let $\ol{\Delta}_1$ be the connection Laplacian on $T^*\ol{M}$, and let $x^i$ be coordinates on $\wh{M}$.
	Because $\ol{\nabla}_rdr=0$ and $\ol{\nabla}_r(r dx^i)=0$ the natural identification $T^*\ol{M}|_{\SetDefine{1}{}\times \wh{M}}\cong T^*\ol{M}|_{\SetDefine{r}{}\times \wh{M}}$ 
	is given by
	\begin{align*}
		\omega_rdr+\omega_i dx^i\sim 	\omega_rdr+r\omega_i dx^i.
	\end{align*}
	It is easy to see that this identification is compatible with the natural connection on $T^*\ol{M}$. Thus
	by Lemma~\ref{lem:conical_connections}, $\ol{\Delta}_1$ is a conical operator.
 \end{exmp}
\begin{exmp}
	Let $\ol{\Delta}_L$ be the Lichnerowicz Laplacian on $S^2\ol{M}$. The identification $S^2\ol{M}|_{\SetDefine{1}{}\times \wh{M}}\cong S^2\ol{M}|_{\SetDefine{r}{}\times \wh{M}}$ coming from parallel transport is given by
	\begin{align*}
		h_{rr}dr\otimes dr+ h_{ri}dr\odot dx^i +h_{ij} dx^i\otimes dx^j\sim 
		h_{rr}dr\otimes dr+r h_{ri}dr\odot dx^i +r^2h_{ij} dx^i\otimes dx^j.	
	\end{align*}
	Again, checking that this identification is compatible with the natural connection on $S^2\ol{M}$ is straightforward. Furthermore, the curvature potential $\mathring{\ol{R}}\in C^{\infty}(\End(S^2\ol{M}))$ scales correctly under that identification. Thus by Lemma~\ref{lem:conical_Schroedinger_operators}, $\ol{\Delta}_L$ is a conical operator. We denote its tangential operator by $\Box_L$.
\end{exmp}

\subsection{Harmonic sections of conical operators}
Recall the indicial root functions $\xi_\pm$ from \eqref{eq:indicial_roots} and let us introduce
\begin{align*}
	\eta&:\C\to\C,\qquad \eta(x)=x(x+n-2).
\end{align*}
By definition,
\begin{align*}
	\xi_{\pm}(x)\in \R	\qquad \forall x\in \left[-\frac{(n-2)^2}{4},\infty\right),
\end{align*}
and $\eta$ restricts to a function $\eta:\R\to\R$.
Furthermore, it is easy to see that
\begin{alignat*}{99}
	\eta(\xi_{\pm}(x))&=x, \qquad& \xi_+(x)\xi_-(x) &= -x &\qquad& \forall x\in \R,\\
	\xi_+(\eta(x))&=x, & \xi_-(\eta(x))&=x+2-n&& \forall x\in \R, \left[\frac{2-n}{2},\infty\right),\\
	\xi_-(\eta(x))&=x, & \xi_+(\eta(x))&=x+n-2&& \forall x\in \left(-\infty,\frac{2-n}{2}\right], \\
	\xi_+(x)+\xi_-(x)&= 2-n  \qquad& \xi_+(x) - \xi_-(x) &= \sqrt{(n-2)^2+4x}&& \forall x\in\R.
\end{alignat*}
The following Lemma is obvious from \eqref{eq:conical_structure}:
\begin{lem}
	Let $\ol{\Delta}_{\ol{V}}$ be a conical self-adjoint Laplace type operator with tangential operator $\Box_{\wh{V}}$. If $u\in C^{\infty}(\wh{V})$ satisfies $\Box_{\wh{V}} u=\nu u$ for some $\nu\in\R$, then 
	\begin{align*}
		\ol{\Delta}_{\ol{V}}(r^{\xi_{\pm}(\nu)}u)=0.
	\end{align*}
	If $\nu=-\frac{(n-2)^2}{4}$, then $\xi_{+}(\nu)=\xi_{-}(\nu)=-\frac{n-2}{2}$ and we also have
	\begin{align*}
		\ol{\Delta}_{\ol{V}}(r^{-\frac{n-2}{2}}\log(r)u)=0.
	\end{align*}
	Conversely, if $\ol{u}\in \mathrm{ker}(\ol{\Delta}_{\ol{V}})$ is of the form $\ol{u}=r^{\alpha}u$ or $\ol{u}=r^{\alpha}\log(r)u$, with $u\in C^{\infty}(E)$, then we have 
	\begin{align*}
		\Box_{\wh{V}}(u)= \eta(\alpha)u.
	\end{align*}
\end{lem}
\begin{defn}
	An element $\alpha\in\SetDefine{\xi_{\pm}(\nu)}{\nu\in\spectrum(\ol{\Delta}_{\ol{V}})}$ is called an \emph{indicial root} of $\ol{\Delta}_{\ol{V}}$.
\end{defn}
Note that self-adjointness and ellipticity of $\ol{\Delta}_{\ol{V}}$ get inherited to $\Box_{\wh{V}}$ and therefore $\spectrum(\Box_{\wh{V}})$ consists solely of eigenvalues of finite multiplicity which tend to infinity.
In particular only a finite number of eigenvalues can be negative, a fact we will use later.
Let now $\SetDefine{\nu_i}{i\in\N}$ be the eigenvalues of $\Box_{\wh{V}}$, counted with multiplicity and $\SetDefine{u_i\in C^{\infty}(\wh{V})}{i\in \N}$ be an orthonormal basis of $L^2(\wh{V})$ such that $\Box_{\wh{V}}(u_i)=\nu_i\cdot u_i$.
Then every $\ol{u}\in \kernel(	\ol{\Delta}_{\ol{V}})$ can be expanded as
\begin{align}\label{eq:harmonic_expansion}
	\ol{u}=\sum_{\substack{i=1\\ \nu_i\neq -\frac{(n-2)^2}{4}}}^{\infty}\left(a_i r^{\xi_{+}(\nu_i)}+b_ir^{\xi_{-}(\nu_i)}\right)u_i+
	\sum_{\substack{j=1\\ \nu_j= -\frac{(n-2)^2}{4}}}^{\infty}\left(a_j r^{-\frac{n-2}{2}}+b_jr^{-\frac{n-2}{2}}\log(r)\right)u_j,
\end{align}
with constants $a_i,a_j,b_i,b_j$, which may have to be
complex in order to ensure that $\ol{u}$ is real. Note that for $\nu_i< -\frac{(n-2)^2}{4}$, we have
\begin{align*}
	\mathrm{Re}(a_i r^{\xi_{+}(\nu_i)}+b_ir^{\xi_{-}(\nu_i)})= r^{\mathrm{Re}(\xi_{+}(\nu_i))}(c_i\cos(\log(\mathrm{Im}\xi_{+}(\nu_i))+d_i \sin(\log(\mathrm{Im}\xi_{+}(\nu_i))))
\end{align*}
for some constants $c_i,d_i\in\R$.
Now, let 
\begin{align*}
	\xi^{\wh{V}}_+:=\min\,\SetDefine{\mathrm{Re}(\xi_{\pm}(\nu_i))}{i\in\N}\cap (0,\infty),\qquad 	\xi^{\wh{V}}_-:=\min\,\SetDefine{-\mathrm{Re}(\xi_{\pm}(\nu_i))}{i\in\N}\cap (0,\infty).
\end{align*}
A straightforward consequence for the boundary behavior of harmonic sections is the following:
\begin{cor}\label{cor:decay_harmonic_sections}
	Let $\ol{u}\in \mathrm{ker}(\ol{\Delta}_{\ol{V}})$, not necessarily defined on all of $\ol{M}$.
	\begin{itemize}
		\item[(i)] If $\ol{u}$ is defined on $(0,\epsilon)\times \wh{M}$ and $|\ol{u}|\to 0$ as $r\to 0$, then $\ol{u}=\O_{\infty}(r^{\xi^{\wh{V}}_+})$ as $r\to0$.
		\item[(ii)] If $\ol{u}$ is defined on $(R,\infty)\times \wh{M}$ and $|\ol{u}|\to 0$ as $r\to\infty $, then 
		we divide into two subcases:
		\begin{itemize}
			\item[(iia)] If $\NotResonanceDominated{\Box_{\wh{V}}}$, then $\ol{u}=\O_{\infty}(r^{-\xi^{\wh{V}}_-})$  as $r\to\infty $.
			\item[(iib)] If $\ResonanceDominated{\Box_{\wh{V}}}$, then $\ol{u}=\O_{\infty}(r^{-\frac{n-2}{2}}\log(r))$  as $r\to\infty $.
		\end{itemize}
	\end{itemize}
\end{cor}
\begin{figure}[tbh!]
	\centering
	\includegraphics{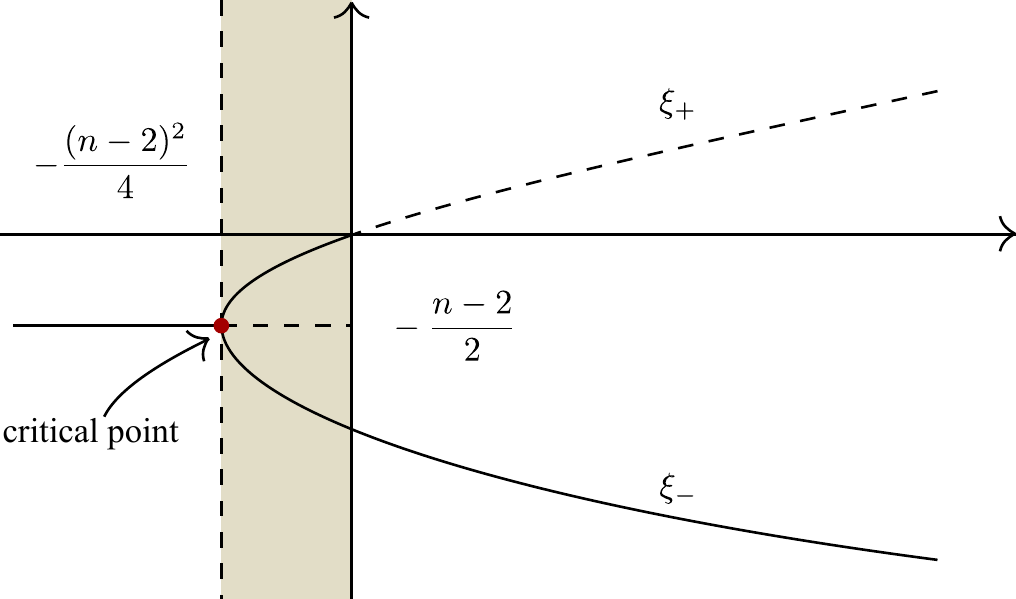}
	\caption[Decay rate of decaying harmonic fields]{The decay rate of decaying harmonic fields at infinity (cf. Corollary~\ref{cor:decay_harmonic_sections}). The figure shows the real part of the indicial roots $\xi_\pm$ corresponding to eigenvalues, cf. \eqref{eq:indicial_roots}. Note that the two branches $\xi_\pm$ coincide at the critical point $-\frac{(n-2)^2}{4}$ and here a logarithmic factor arises. However, if there is another eigenvalue in the shaded region then a slower decay rate dominates the logarithmic one at infinity.}
	\label{fig:generic-decay}
\end{figure}

\subsection{The tangential operator of the connection Laplacian \texorpdfstring{$\ol{\Delta}_1$}{}}\label{subsec:tang_op_forms}
We define
\begin{align}
	D(\wh{M}):=\SetDefine{\omega\in C^{\infty}(T^*\wh{M})}{\wh{\delta}\omega=0}.
\end{align}
As it is well known from the Hodge decomposition, we have the $L^2$-orthogonal splitting
\begin{align}\label{eq:decomp_1forms}
	C^{\infty}(T^*\wh{M})=d(C^{\infty}(\wh{M}))\oplus D(\wh{M}),
\end{align}
since $\wh{M}$ is compact.
Due to the commutation rules \eqref{commutation},
this splitting is preserved by $\wh{\Delta}_1$ and we have
\begin{align}
	\spectrum(\wh{\Delta}_1)=\spectrum(\wh{\Delta}_0-(n-2))\cup\spectrum(\wh{\Delta}_1|_{D(\wh{M})}).
\end{align}
Let $0=\lambda_0<\lambda_1<\lambda_2\ldots$ be the eigenvalues of $\wh{\Delta}$  and
let $\mu_1<\mu_2<\mu_3\ldots$ be the eigenvalues of $\wh{\Delta}_1|_{D(\wh{M})}$. 
A standard calculation shows that $\wh{\Delta}_1=2\wh{\delta}\circ\wh{\delta}^*+(n-2)$ holds on $D(\wh{M})$. In particular, $\mu_i\geq (n-2)$ for all $i\in\N$ and $\wh{\Delta}_1\omega=(n-2)\omega$ holds if and only if $\wh{\delta}^*\omega=0$, i.e. if $\omega^{\sharp}$ is a Killing vector field.
To emphasize this relation to Killing vector fields, we always start to count the $\mu_i$ from $0$, whenever $(n-2)\in\spectrum(\wh{\Delta}_1|_{D(\wh{M})})$, i.e. we set $\mu_0=(n-2)$.
\begin{lem}\label{hatdiv}
	Let $\omega\in D(\wh{M})$, ${\varphi}\in C^{\infty}(0,\infty)$ and $\ol{\omega}\in C^{\infty}(T^*\ol{M})$ be given by $\ol{\omega}=\varphi\cdot r\omega$. Then,
	\begin{align}\label{divfree1formsA1}
		\ol{\delta}\ol{\omega}=0,\qquad \ol{\omega}(\partial_r)=0
	\end{align}
	and
	\begin{align}\label{divfree1formsA2}
		\ol{\Delta}_1\ol{\omega}=r(-\partial^2_{rr}{\varphi}\cdot\omega-(n-1)r^{-1}\partial_r{\varphi}\cdot\omega+{\varphi}\cdot r^{-2}(\wh{\Delta}_1+1)\omega).
	\end{align}
\end{lem}
\begin{proof}
	If $\omega $ and $\ol{\omega}$ are as in the statement, \eqref{christoffelA2} implies
	\begin{align}
		\ol{\nabla}_i\ol{\omega}_{j}=\varphi r\wh{\nabla}_i\omega_j,\qquad \ol{\nabla}_r\ol{\omega}_j=\partial_r\varphi\cdot r\omega_j,\qquad 	\ol{\nabla}_i\ol{\omega}_{r}=-\varphi\cdot\omega_j,\qquad \ol{\nabla}_r\ol{\omega}_{r}=0,
	\end{align}
	and \eqref{divfree1formsA1} follows by taking the trace and the fact that $\ol{\omega}_r=0$. By applying the covariant derivative once again, we obtain from \eqref{christoffelA2} that
	\begin{alignat}{99}
		\ol{\nabla}_{ij}^2\ol{\omega}_k&=\varphi\cdot r\wh{\nabla}_{ij}^2\omega_k+r^2\wh{g}_{ij}\partial_r\varphi\cdot\omega_k-r\varphi \wh{g}_{jk}\omega_i,&\qquad&	\ol{\nabla}_{rr}^2\ol{\omega}_k=r\partial^2_{rr}\varphi\cdot\omega_k,\nonumber \\
		\ol{\nabla}_{ij}^2\ol{\omega}_r&=-\varphi(\wh{\nabla}_i\omega_j+\wh{\nabla}_j\omega_i)=-2\varphi\cdot(\wh{\delta}^*\omega)_{ij},&\qquad&
		\ol{\nabla}_{rr}^2\ol{\omega}_r=0
	\end{alignat}
	and \eqref{divfree1formsA2} follows by taking the trace and using that $\wh{\delta} \omega=0$.
\end{proof}
\begin{prop}\label{prop:spectrum_tang_1forms}
	The spectrum of the tangential operator ${\Box}_1$ is given by
	\begin{align*}
		\spectrum({\Box}_1)=\SetDefine{\mu_i+1}{i\in \N}\cup \SetDefine{\lambda^{(1)}_{i,\pm}}{i\in \N }\cup\SetDefine{\lambda^{(1)}_{0,-}}{},
	\end{align*}
	where
	\begin{align*}
		\lambda^{(1)}_{i,\pm}:=\eta(\xi_{\pm}(\lambda)-1).
	\end{align*}
	The corresponding eigenspaces are
	\begin{align*}
		E({\Box}_1,\mu_i+1)&=\SetDefine{\omega}{\omega \in E(\wh{\Delta}_1,\mu_i) },\\
		E({\Box}_1,\lambda^{(1)}_{i,\pm})&=\SetDefine{\xi_{\pm}(\lambda)vdr+dv}{v \in E(\wh{\Delta},\lambda_i) },\\
		E({\Box}_1,\lambda^{(1)}_{0,-}=n-1)&=\SetDefine{\alpha dr}{\alpha\in\R},
	\end{align*}
	where $i\in\N$.
\end{prop}
\begin{rem}
	The indicial values of $\ol{\Delta}$ and $\ol{\Delta}_1$ coming from an eigenfunction $v\in E(\wh{\Delta},\lambda_i)$ are related as follows:
\begin{align*}
	\xi_+(\lambda_i)=\xi_+(\lambda^{(1)}_{i,+})+1=\xi_+(\lambda^{(1)}_{i,-})-1,\qquad 
	\xi_-(\lambda_i)=\xi_-(\lambda^{(1)}_{i,-})+1=\xi_-(\lambda^{(1)}_{i,+})-1.
\end{align*}	
\end{rem}
\begin{proof}[Proof of Proposition~\ref{prop:spectrum_tang_1forms}]
	Let $\omega\in D(\wh{M})$ be such that $\wh{\Delta}_1\omega=\mu\omega$. Then due to Lemma~\ref{hatdiv}, we have
	\begin{align*}
		r^{\xi_{\pm}(\mu+1)}(r\omega)\in\ker(\ol{\Delta}_1).
	\end{align*}	
	On the other hand, because $\ol{\Delta}_1$ is a conical operator, we have
	\begin{align*}
		\Box_1\omega=\eta(\xi_{\pm}(\mu+1))\omega=(\mu+1)\omega,
	\end{align*}
	which gives the first type of eigenvalues and eigenforms.
	Now, take $v\in C^{\infty}(\wh{M})$ with $\wh{\Delta} v=\lambda v$. Then,
	\begin{align*}
		r^{\xi_{\pm}(\lambda)}v\in \ker(\ol{\Delta}),
	\end{align*}
	and because $d\circ \ol{\Delta}=\ol{\Delta}_1\circ d$, we also get
	\begin{align*}
		d(r^{\xi_{\pm}(\lambda)}v)=r^{\xi_{\pm}(\lambda)-1}(\xi_{\pm}(\lambda)vdr+rdv)\in \ker(\ol{\Delta}_1)
	\end{align*}
	because $(\ol{M},\ol{g})$ is Ricci-flat. Since $\ol{\Delta}_1$ is a conical operator, we conclude that
	the section
	\begin{align*}
		\xi_{\pm}(\lambda)vdr+dv\in C^{\infty}(T^*\ol{M}|_{\SetDefine{1}{}\times \wh{M}})
	\end{align*}
	satisfies
	\begin{align*}
		\Box_1(\xi_{\pm}(\lambda)vdr+dv)=\lambda^{(1)}_{i,\pm}(\xi_{\pm}(\lambda)vdr+dv),\qquad \lambda^{(1)}_{i,\pm}:= \eta(\xi_{\pm}(\lambda)-1).
	\end{align*}
	For $\lambda>0$, observe that the reals $\xi_{\pm}(\lambda)$ and the section $dv$  are all nonvanishing. For $\lambda=\lambda_0=0$, $v\equiv c\in\R$ and $\xi_+(0)=0$, $\xi_-(0)=1-n$. Hence,
	\begin{align*}
		\xi_{+}(\lambda)vdr+dv=0,\qquad \xi_{-}(\lambda)vdr+dv=(1-n)cdr,\qquad \lambda_{0,-}=0.
	\end{align*}
	We have now constructed all eigenvalues and eigenspaces stated in the proposition. Since the $L^2$-span of the $E(\wh{\Delta},\lambda_i)$ is already all of $L^2(\wh{M})$ and the $L^2$-span of $E(\wh{\Delta}_1,\mu_i)$ is all of $L^2(D(\wh{M}))$, it follows with the help of \eqref{eq:decomp_1forms} that the $L^2$-span of the eigenspaces for $\Box_1$ we have constructed so far is already all of 
	\begin{align*}
		L^2(T^*\ol{M}|_{\SetDefine{1}{}\times \wh{M}})\cong L^2(\wh{M})\oplus L^2(T^*\wh{M}).
	\end{align*}
	This finishes the proof of the lemma.
\end{proof}

\subsection{The tangential operator of the Lichnerowicz Laplacian \texorpdfstring{$\ol{\Delta}_L$}{}}
Let us assume in this subsection that the dimension of $\wh{M}$ is $n-1\geq3$.
We have the $L^2$-orthogonal splitting
\begin{align}\label{eq:decomp_2tensors}
	C^{\infty}(S^2\wh{M})=C^{\infty}(\wh{M})\cdot \wh{g}\oplus \SetDefine{n\wh{\nabla}^2v+\wh{\Delta} v\cdot \wh{g}}{v\in C^{\infty}(\wh{M}) }\oplus \wh{\delta}^*(D(\wh{M}))\oplus TT(\wh{M}),
\end{align}
where $TT(\wh{M})=\SetDefine{h\in C^{\infty}(S^2\wh{M})}{\wh{\trace} h=0,\wh{\delta} h=0}$ denotes the space of transverse traceless tensors on $\wh{M}$.
The Einstein operator \eqref{eq:Einstein-operator} has a block diagonal form with respect to this decomposition \cite[p. 130]{Bes08} (see  also \cite[Section~2]{Kro17} for the refined version stated here). 
\begin{lem}\label{ttlemma}
	Let $h\in TT(\wh{M})$, $\varphi\in C^{\infty}((0,\infty))$
	and $\ol{h}\in C^{\infty}(S^2\ol{M})$ be defined by $\ol{h}=\varphi\cdot r^2h$. This tensor satisfies
	\begin{align}
		\ol{\trace}\ol{h}=0,\qquad \ol{\delta}\ol{h}=0,\qquad \ol{h}(\partial_r,.)=0
	\end{align}
	and
	\begin{align}\label{eq:TT_conical}
		\ol{\Delta}_L\ol{h}=r^2(-\partial^2_{rr}\varphi\cdot h-(n-1)\cdot r^{-1}\partial_r\varphi\cdot h+\varphi\cdot r^{-2}\wh{\Delta}_Eh).
	\end{align}
\end{lem}
\begin{proof}
	The condition $\ol{h}(\partial_r,.)=0$ holds as $\ol{h}_{rr}=\ol{h}_{rj}=0$.
	This immediately implies $\ol{\trace}\ol{h}=0$ as $\wh{\trace}_{\wh{g}}h=0$.  By using \eqref{christoffelA2},
	\begin{equation}
		\begin{split}
			\ol{\nabla}_i\ol{h}_{jk}&=\varphi r^2\wh{\nabla}_ih_{jk},\qquad \ol{\nabla}_rh_{ij}=\partial_r\varphi\cdot r^2h_{ij},\qquad \ol{\nabla}_i\ol{h}_{jr}=\ol{\nabla}_i\ol{h}_{rj}=-\varphi\cdot r\cdot h_{ij},\\
			\ol{\nabla}_i\ol{h}_{rr}&=\ol{\nabla}_r\ol{h}_{jr}=\ol{\nabla}_r\ol{h}_{rk}=0,
		\end{split}
	\end{equation}
	and by taking the trace with respect to $\ol{g}$ and using $\wh{\trace}h=0$, we obtain $ \ol{\delta}\ol{h}=0$. Taking the covariant derivative once again, we obtain
	\begin{equation}
		\begin{split}
			\ol{\nabla}^2_{ij}\ol{h}_{kl}&=\varphi r^2\cdot \wh{\nabla}^2_{ij}h_{kl}-r\wh{g}_{ij}\partial_r\varphi\cdot r^2h_{kl}+r^2(\wh{g}_{ik}h_{jl}+\wh{g}_{il}h_{jk}),\\
			\ol{\nabla}^2_{rr}\ol{h}_{kl}&=\partial^2_{rr}\varphi \cdot r^2h_{kl},\\
			\ol{\nabla}^2_{ij}\ol{h}_{rr}&=	2\varphi\cdot h_{ij},\\
			\ol{\nabla}^2_{rr}\ol{h}_{rr}&=	\ol{\nabla}^2_{rr}\ol{h}_{kr}=	\ol{\nabla}^2_{rr}\ol{h}_{rl}=	0,\\
			\ol{\nabla}^2_{ij}\ol{h}_{kr}&=	\ol{\nabla}^2_{ij}\ol{h}_{rk}=-2\varphi r(\wh{\nabla}_ih_{jk}+\wh{\nabla}_jh_{ik}).
		\end{split}
	\end{equation}
	By taking the trace and using that $h\in TT(\wh{M})$, we obtain
	\begin{align}
		\ol{\Delta}_2\ol{h}=r^2(-\partial^2_{rr}\varphi\cdot h-(n-1)\cdot r^{-1}\partial_r\varphi\cdot h+\varphi\cdot r^{-2}\wh{\Delta}_2h).	
	\end{align}
	It remains to consider the curvature term. However, the only nonvanishing term of the curvature of $\ol{g}$ is
	\begin{align}
		\ol{R}_{ijkl}=r^2(\wh{R}_{ijkl}+\wh{g}_{ik}\wh{g}_{jl}-\wh{g}_{il}\wh{g}_{jk}),
	\end{align}
	therefore
	\begin{align}
		\mathring{\ol{R}}(\ol{h})_{ij}=\varphi\mathring{\wh{R}}(h)_{ij},
	\end{align}
	which by adding up finishes the proof of the lemma.
\end{proof}
\begin{lem}\label{lem:conical_liediv}
	Let $\varphi,\psi\in C^{\infty}((0,\infty))$, $\omega\in C^{\infty}(T^*\wh{M})$ and $v\in C^{\infty}(\wh{M})$.
	Then, the form
	\begin{align*}
		\ol{\omega}=\varphi r\omega+\psi vdr\in C^{\infty}(T^*\ol{M})
	\end{align*}
	satisfies
	\begin{align*}
		\ol{\delta}^*\ol{\omega}=r^{-1}\varphi(r^2 \wh{\delta}^*\omega)+r^{-1}\psi v(r^2\wh{g})+\partial_r\psi \cdot v dr\otimes dr +\frac{1}{2}[(\partial_r\varphi-r^{-1}\varphi)(r\omega)+r^{-1}\psi (r dv)]\odot dr.
	\end{align*}
\end{lem}
\begin{proof}
	Straightforward calculations show that
	\begin{align*}
		\ol{\nabla}_i\ol{\omega}_j&=r\varphi\wh{\nabla}_i\omega_j+r\wh{g}_{ij}\psi v, &
		\ol{\nabla}_r\ol{\omega}_r&=\partial_r\psi \cdot v,\\
		\ol{\nabla}_r\ol{\omega}_j&=r\partial_r\varphi\cdot \omega_j, &
		\ol{\nabla}_i\ol{\omega}_r&=\psi\cdot\partial_iv-\varphi\omega_i,
	\end{align*}
	and the result is immediate.
\end{proof}

\begin{thm}\label{thm_spectrum_tangential_LL}
	The spectrum of the tangential operator $\Box_L$ of the Lichnerowicz Laplacian is given by
	\begin{align*}
		\spectrum(\Box_L)=\SetDefine{\kappa_i}{i\in \N}\cup \SetDefine{\mu^{(1)}_{i,\pm}}{i\in \N}\cup \SetDefine{\lambda_{i}}{i\in \N}\cup
		\SetDefine{\lambda^{(2)}_{i,\pm}}{i\in \N}
		\cup\SetDefine{0,2n}{},
	\end{align*}
	where
	\begin{align*}
		\mu^{(1)}_{i,\pm}:=\eta(\xi_{\pm}(\mu_i+1)-1),\qquad \lambda^{(2)}_{i,\pm}:=\eta(\xi_{\pm}(\lambda_i)-2).
	\end{align*}
	The corresponding eigenspaces are
	\begin{align*}
		E(\Box_L,\kappa_i)&=\SetDefine{h}{h \in E(\wh{\Delta}_E,\kappa_i)},\\
		E(\Box_L,\mu^{(1)}_{i,\pm})&=\SetDefine{\wh{\delta}^*\omega+\frac{1}{2}(\xi_{\pm}(\mu_i+1)-1)\omega\odot dr}{\omega \in E(\wh{\Delta}_1,\mu_i)},\\
		E(\Box_L,\mu^{(1)}_{0,-}=2n)&=\SetDefine{\omega\odot dr}{\omega \in E(\wh{\Delta}_1,\mu_i) },\\
		E(\Box_L,\lambda^{(2)}_{i,\pm})&=\SetDefine{\mathring{\wh{\nabla}}^2v-\left(\frac{\lambda_i}{n-1}-\xi_{\pm}(\lambda_i)\right) v\mathring{\wh{g}}+(\xi_{\pm}(\lambda_i)-1)\wh{\nabla} v\odot dr}{v \in E(\wh{\Delta},\lambda_i) },\\
		E(\Box_L,\lambda_{i})&=\SetDefine{\mathring{\wh{\nabla}}^2v
			+\frac{\lambda_i(n-2)}{n(n-1)}v\mathring{\wh{g}}-\frac{1}{2}(n-2)\wh{\nabla} v\odot dr+w\cdot \ol{g}}{v,w \in E(\wh{\Delta},\lambda_i) }, \\
		E(\Box_L,\lambda_0=0)&=\SetDefine{\alpha \ol{g}}{\alpha\in\R},\\
		E(\Box_L,\lambda^{(2)}_{0,-}=2n)&=\SetDefine{\alpha \mathring{\wh{g}}}{\alpha\in\R}.
	\end{align*}
	Here, $\mathring{\wh{\nabla}}^2$ denotes the trace-free part of the Hessian \eqref{eq:trace-free-Hessian} and $\mathring{\wh{g}}=\wh{g}-(n-1)dr\otimes dr$ denotes the trace-free part of $\wh{g}$ with respect to $\ol{g}$.
\end{thm}
\begin{rem}
	If $(\wh{M},\wh{g})$ admits Killing fields, then we start to count the $\mu_i$ from zero and $\mu_0=(n-2)$, c.f.\ the discussion at the beginning of Subsection \ref{subsec:tang_op_forms}. In this case, $\xi_+(\mu_0+1)=1$ and the eigenvalue $\mu_{0,+}$ drops as the corresponding eigentensor vanishes. On the other hand $\mu^{(1)}_{0,-}$ still exists and equals $\mu^{(1)}_{0,-}=\lambda^{(2)}_{0,-}=2n$. Thus by this counting convention, we do not need to distinguish between the cases with or without Killing fields in Theorem~\ref{mainthm:indicial_roots_LL} and always count $\mu_i$ for $i\in\N$ there.
	
	Similarly by the Lichnerowicz--Obata eigenvalue inequality \cite{Ob62}, we have $\lambda_i\geq n-1$ for all $i\in\N$ and equality holds only for the round sphere. In this case, we have $\mathring{\wh{\nabla}}^2v=0$ for the corresponding eigenfunctions. 
		 In this case $\xi_+(\lambda_{1})=1$ and the eigenvalue 
	$\lambda^{(2)}_{1,+}$ drops and the eigenvalue $\lambda^{(2)}_{1,-}=3(n+1)$ is still present.
\end{rem}
\begin{rem}
	The indicial values of $\ol{\Delta}$ and $\ol{\Delta}_L$ coming from an eigenfunction $v\in E(\wh{\Delta},\lambda_i)$ are related by
	\begin{align*}
		\xi_+(\lambda_i)=\xi_+(\lambda^{(2)}_{i,+})+2=\xi_+(\lambda^{(2)}_{i,-})-2,\qquad 
		\xi_-(\lambda_i)=\xi_-(\lambda^{(2)}_{i,-})+2=\xi_-(\lambda^{(2)}_{i,+})-2
	\end{align*}
	and indicial values of $\ol{\Delta}_1$ and $\ol{\Delta}_L$ coming from an eigenform $\omega\in E(\wh{\Delta}_1|_{D(\wh{M})},\mu_i)$	are related by
	\begin{align*}
		\xi_+(\mu_i)=\xi_+(\mu^{(1)}_{i,+})+1=\xi_+(\mu^{(1)}_{i,-})-1,\qquad 
		\xi_-(\mu_i)=\xi_-(\mu^{(1)}_{i,-})+1=\xi_-(\mu^{(1)}_{i,+})-1.
	\end{align*}
\end{rem}
\begin{proof}[Proof of Theorem~\ref{thm_spectrum_tangential_LL}]
	Let $h\in TT(\wh{M})$ be such that $\wh{\Delta}_Eh=\kappa h$. Then due to Lemma~\ref{ttlemma}, we have
	\begin{align*}
		r^{\xi_{\pm}(\kappa)}(r^2h)\in\ker(\ol{\Delta}_L).
	\end{align*}	
	Comparing \eqref{eq:TT_conical} and the conical structure of $\ol{\Delta}_L$, we however immediately get 
	\begin{align*}
		\Box_Lh=\wh{\Delta}_Eh=\kappa h.
	\end{align*}	
	Next let $\omega\in D(\wh{M})$ be such that $\wh{\Delta}_1\omega=\kappa\omega$. From Lemma~\ref{hatdiv}, we know that
	\begin{align*}
		\ol{\omega}=r^{\xi_{\pm}(\mu+1)}(r\omega)\in \ker(\ol{\Delta}_1).
	\end{align*}
	From Lemma~\ref{lem:conical_liediv} and the commutation rules \eqref{commutation},
	\begin{align*}
		\ol{\delta}^*\ol{\omega}=r^{\xi_{\pm}(\mu+1)-1}\left[(r^2\wh{\delta}^*\omega)+\frac{1}{2}(\xi_{\pm}(\mu+1)-1)(r\omega)\odot dr\right]\in \ker(\ol{\Delta}_L).
	\end{align*}
	Because the Einstein operator is conical, we get that the sections
	\begin{align}\label{eq:h_pm}
		h_{\pm}:=\wh{\delta}^*\omega+\frac{1}{2}(\xi_{\pm}(\mu+1)-1)\omega\odot dr\in C^{\infty}(S^2\ol{M}|_{\SetDefine{1}{}\times\wh{M}})
	\end{align}
	satisfy
	\begin{align*}
		\Box_L(h_{\pm})=\eta(\xi_{\pm}(\mu+1)-1)h_{\pm}=\mu_{\pm}^{(1)}h_{\pm}.
	\end{align*}
	To conclude the discussions on eigenvalues  and eigensections generated by elements in $D(\wh{M})$, we recall that $\mu\geq n-2$ and $\wh{\Delta}_1\omega=(n-2)\omega$ if and only if $\wh{\delta}^*\omega=0$. In this case, $\xi_{+}(\mu+1)-1=\xi_{+}(n-1)-1=0$, therefore $h_+=0$ and $\xi_{-}(\mu+1)-1=\xi_{-}(n-1)-1=-n$. This implies $h_-\in \R\cdot \omega\odot dr$ and $h_-\neq 0$ if $\omega\neq 0$.
	
	It remains to compute the eigenvalues and eigensections generated by smooth functions on $\wh{M}$. Let $w\in C^{\infty}(\wh{M})$ be such that $\wh{\Delta} w=\lambda w$. Then,
	\begin{align*}
		r^{\xi_{\pm}(\lambda)}w\in \ker(\ol{\Delta}),
	\end{align*}
	and hence
	\begin{align*}
		r^{\xi_{\pm}(\lambda)}w\cdot \ol{g}=r^{\xi_{\pm}(\lambda)}w(dr\otimes dr+r^2\wh{g})\in \ker(\ol{\Delta}_L).
	\end{align*}
	Due to the structure of the Lichnerowicz Laplacian, we get for
	\begin{align*}
		h_{1}(w):=r^{\xi_{\pm}(\lambda)}w\cdot \ol{g}|_{\SetDefine{1}{}\times \wh{M}}=w(dr\otimes dr+\wh{g})
	\end{align*}
	that
	\begin{align*}
		\Box_Lh_{1}(w)=\lambda\cdot  h_{1}(w).
	\end{align*}
	Finally, it remains to consider the $\ol{g}$-traceless eigensections generated by functions. For $v\in C^{\infty}(\wh{M})$ with $\wh{\Delta}_0 v = \lambda v$ and $\alpha,\beta\in\R$, we first make the ansatz
	\begin{align*}
		\ol{\omega}=r^{\beta}(\alpha v dr+rdv).
	\end{align*}
	Using Lemma~\ref{lem:conical_liediv}, we get
	\begin{align*}
		\ol{\delta}^*(r^\beta vdr)&=r^{\beta-1}(vr^2\wh{g}+\beta vdr\otimes dr+\frac{1}{2}r\wh{\nabla} v\odot dr),\\
		\ol{\delta}^*(r^\beta rdv)&=r^{\beta-1}(r^2\wh{\nabla}^2v+\frac{1}{2}(\beta-1)r\wh{\nabla} v\odot dr),
	\end{align*}
	and we conclude the trace-free part of the symmetrized covariant derivative equals
	\begin{align*}
		\mathring{\ol{\delta}}^*\ol{\omega}
		&=\ol{\delta}^*\ol{\omega}+\frac{1}{n}\ol{\delta}\ol{\omega}\\
		&=r^{\beta-1}\Bigg[r^2\wh{\nabla}^2v+\frac{\lambda}{n}v(dr\otimes dr+r^2\wh{g})+\frac{\alpha(\beta-1)}{n}v((n-1)dr\otimes dr-r^2\wh{g}) \\
		&\qquad +\frac{1}{2}(\alpha+\beta-1)r\wh{\nabla} v\odot dr\Bigg]\\
		&=r^{\beta-1}\Bigg[r^2\left(\wh{\nabla}^2v+\frac{\lambda}{n-1}v\wh{g}\right) 
		+\frac{\lambda+(n-1)\alpha(\beta-1)}{n(n-1)}v((n-1)dr\otimes dr-r^2\wh{g}) \\
		&\qquad +\frac{1}{2}(\alpha+\beta-1)r\wh{\nabla} v\odot dr\Bigg].
	\end{align*}
	We now want to find $\alpha,\beta \in\R$ such that  $\ol{\omega}\in \ker(\ol{\Delta}_1)$.
	By Proposition~\ref{prop:spectrum_tang_1forms}, we have
	\begin{align*}
		\Box_1(\xi_{\pm}(\lambda)vdr+dv)=\lambda^{(1)}_{i,\pm}(\xi_{\pm}(\lambda)vdr+dv),\qquad \lambda^{(1)}_{\pm}:= \eta(\xi_{\pm}(\lambda)-1).
	\end{align*}
	Since the connection Laplacian $\wh{\Delta}_1$ on one-forms is conical, the forms
	\begin{align*}
		&r^{\xi^{(1)}_{+}(\lambda_+)}(\xi_{+}(\lambda)vdr+rdv),\qquad
		&r^{\xi^{(1)}_{-}(\lambda_+)}(\xi_{+}(\lambda)vdr+rdv),\\
		&r^{\xi^{(1)}_{+}(\lambda_-)}(\xi_{-}(\lambda)vdr+rdv),\qquad
		&r^{\xi^{(1)}_{-}(\lambda_-)}(\xi_{-}(\lambda)vdr+rdv),
	\end{align*}
	are all harmonic on $\ol{M}$. 
	Due to the previous computations, the trace-free parts of their symmetrized covariant derivatives are given by the tensors
	\begin{align*}
		\ol{h}_2(v):&=r^{\xi_{+}(\lambda_+)-1}\Bigg[r^2\left(\wh{\nabla}^2v+\frac{\lambda}{n-1}v\wh{g}\right) 
		+\tfrac{\lambda+(n-1)\xi_{+}(\lambda)(\xi_{+}(\lambda^{(1)}_+)-1)}{n(n-1)}v((n-1)dr\otimes dr-r^2\wh{g})\\
		&\qquad+\frac{1}{2}(\xi_{+}(\lambda)+\xi_{+}(\lambda^{(1)}_+)-1)r\wh{\nabla} v\odot dr\Bigg],\\
		\ol{h}_3(v):&=r^{\xi_{-}(\lambda^{(1)}_-)-1}\Bigg[r^2\left(\wh{\nabla}^2v+\frac{\lambda}{n-1}v\wh{g}\right) 
		+\tfrac{\lambda+(n-1)\xi_{-}(\lambda)(\xi_{-}(\lambda^{(1)}_-)-1)}{n(n-1)}v((n-1)dr\otimes dr-r^2\wh{g})\\
		&\qquad+\frac{1}{2}(\xi_{-}(\lambda)+\xi_{-}(\lambda^{(1)}_-)-1)r\wh{\nabla} v\odot dr\Bigg],\\
		\ol{h}_4(v):&=r^{\xi_{+}(\lambda^{(1)}_-)-1}\Bigg[r^2\left(\wh{\nabla}^2v+\frac{\lambda}{n-1}v\wh{g}\right)
		+\tfrac{\lambda+(n-1)\xi_{-}(\lambda)(\xi_{+}(\lambda^{(1)}_-)-1)}{n(n-1)}v((n-1)dr\otimes dr-r^2\wh{g})\\
		&\qquad+\frac{1}{2}(\xi_{-}(\lambda)+\xi_{+}(\lambda^{(1)}_-)-1)r\wh{\nabla} v\odot dr\Bigg],\\
		\ol{h}_5(v):&=r^{\xi_{-}(\lambda^{(1)}_+)-1}\Bigg[r^2\left(\wh{\nabla}^2v+\frac{\lambda}{n-1}v\wh{g}\right)
		+\tfrac{\lambda+(n-1)\xi_{+}(\lambda)(\xi_{-}(\lambda^{(1)}_+)-1)}{n(n-1)}v((n-1)dr\otimes dr-r^2\wh{g})\\
		&\qquad+\frac{1}{2}(\xi_{+}(\lambda)+\xi_{-}(\lambda^{(1)}_+)-1)r\wh{\nabla} v\odot dr\Bigg],
	\end{align*}
	which are, by the commutation rules \eqref{commutation}, all in the kernel of the Lichnerowicz Laplacian.
	Simple manipulations yield
	\begin{align*}
		\xi_{+}(\lambda^{(1)}_+)&=\xi_{+}(\lambda)-1,& \xi_{-}(\lambda^{(1)}_+)&=\xi_{-}(\lambda)+1,\\
		\xi_{-}(\lambda^{(1)}_-)&=\xi_{-}(\lambda)-1,& \xi_{+}(\lambda^{(1)}_-)&=\xi_{+}(\lambda)+1,
	\intertext{and therefore,}
		\xi_{+}(\lambda^{(1)}_+)-1&=\xi_{+}(\lambda)-2,& \xi_{-}(\lambda^{(1)}_-)-1&=\xi_{-}(\lambda)-2,\\
		\xi_{+}(\lambda^{(1)}_-)-1&=\xi_{+}(\lambda),  & \xi_{-}(\lambda^{(1)}_+)-1&=\xi_{-}(\lambda).
	\end{align*}
	This implies that
	\begin{alignat*}{99}
		\xi_{-}(\lambda)(\xi_{+}(\lambda^{(1)}_-)-1)&=\xi_{+}(\lambda)(\xi_{-}(\lambda^{(1)}_+)-1)&&=\xi_{+}(\lambda)\cdot\xi_{-}(\lambda)&&=-\lambda,\\
		\xi_-(\lambda)+\xi_{+}(\lambda^{(1)}_-)-1&=\xi_{+}(\lambda)+\xi_{-}(\lambda^{(1)}_+)-1&&=\xi_{+}(\lambda)+\xi_{-}(\lambda)&&=2-n.
	\end{alignat*}
	Furthermore, we compute
	\begin{align*}
		\lambda+(n-1)\xi_{\pm}(\lambda)(\xi_{\pm}(\lambda)-2)&=\lambda+(n-1)\xi_{\pm}(\lambda)(\xi_{\pm}(\lambda)+n-2-n)\\
		&=\lambda+(n-1)\lambda-(n-1)n\xi_{\pm}(\lambda)=
		n(\lambda-(n-1)\xi_{\pm}(\lambda)),
	\end{align*}
	therefore
	\begin{align*}
		\ol{h}_2(v)&=r^{\xi_{+}(\lambda)-2}\Bigg[r^2\left(\wh{\nabla}^2v+\frac{\lambda}{n-1}v\wh{g}\right) 
		+\frac{\lambda-(n-1)\xi_{+}(\lambda)}{n-1}v((n-1)dr\otimes dr-r^2\wh{g})\\
		&\qquad+(\xi_{+}(\lambda)-1)r\wh{\nabla} v\odot dr\Bigg],\\
		\ol{h}_3(v)&=r^{\xi_{-}(\lambda)-2}\Bigg[r^2\left(\wh{\nabla}^2v+\frac{\lambda}{n-1}v\wh{g}\right) 
		+\frac{\lambda-(n-1)\xi_{-}(\lambda)}{n-1}v((n-1)dr\otimes dr-r^2\wh{g})\\
		&\qquad+(\xi_{-}(\lambda)-1)r\wh{\nabla} v\odot dr\Bigg],\\
		\ol{h}_4(v)&=r^{\xi_{+}(\lambda)}\left[r^2\left(\wh{\nabla}^2v+\frac{\lambda}{n-1}v\wh{g}\right) 
		-\frac{\lambda(n-2)}{n(n-1)}v((n-1)dr\otimes dr-r^2\wh{g})
		-\frac{1}{2}(n-2)r\wh{\nabla} v\odot dr\right],\\
		\ol{h}_5(v)&=r^{\xi_{-}(\lambda)}\left[r^2\left(\wh{\nabla}^2v+\frac{\lambda}{n-1}v\wh{g}\right) 
		-\frac{\lambda(n-2)}{n(n-1)}v((n-1)dr\otimes dr-r^2\wh{g})
		-\frac{1}{2}(n-2)r\wh{\nabla} v\odot dr\right].
	\end{align*}
	Because the Lichnerowicz Laplacian is conical, we get that the three tensors
	\begin{align*}
		\wh{h}_2(v):&=\ol{h}_2(v)|_{\SetDefine{1}{}\times \wh{M}}&\\
		&=\left(\wh{\nabla}^2v+\frac{\lambda}{n-1}v\wh{g}\right) 
		+\left(\frac{\lambda}{n-1}-\xi_{+}(\lambda)\right)v((n-1)dr\otimes dr-\wh{g})
		+(\xi_{+}(\lambda)-1)\wh{\nabla} v\odot dr,\\
		\wh{h}_3(v):&=\ol{h}_3(v)|_{\SetDefine{1}{}\times \wh{M}}\\
		&=\left(\wh{\nabla}^2v+\frac{\lambda}{n-1}v\wh{g}\right) 
		+\left(\frac{\lambda}{n-1}-\xi_{-}(\lambda)\right)v((n-1)dr\otimes dr-\wh{g})
		+(\xi_{-}(\lambda)-1)\wh{\nabla} v\odot dr,\\
		\wh{h}_4(v):&=\ol{h}_4(v)|_{\SetDefine{1}{}\times \wh{M}}\\
		&=\left(\wh{\nabla}^2v+\frac{\lambda}{n-1}v\wh{g}\right) 
		-\frac{\lambda(n-2)}{n(n-1)}v((n-1)dr\otimes dr-\wh{g})
		-\frac{1}{2}(n-2)\wh{\nabla} v\odot dr \\
		&=\ol{h}_5(v)|_{\SetDefine{1}{}\times \wh{M}}
	\end{align*}
	are eigensections of $\Box_L$ with eigenvalues
	\begin{align*}
		\Box_L \wh{h}_2(v)=\eta(\xi_+(\lambda)-2)\wh{h}_2(v),\qquad \Box_L \wh{h}_3(v)=\eta(\xi_-(\lambda)-2)\wh{h}_2(v),\qquad \Box_L \wh{h}_4(v)=\lambda \wh{h}_4(v).
	\end{align*}
	This finishes the discussion of eigenvalues $\lambda>0$.
	Recall that if $\lambda=0$, one has $\xi_+(\lambda)=0$ and $\xi_-(\lambda)=2-n$; moreover, $v$ is constant and all its covariant derivatives vanish. In this case, the only nonvanishing tensor of the above is
	\begin{align*}
		\wh{h}_3(v)=(n-2)v((n-1)dr\otimes dr-\wh{g})
	\end{align*}
	and the corresponding eigenvalue is $\eta(\xi_-(\lambda)-2)=\eta(-n)=(-n)(-n+n-2)=2n$. By now we have constructed all eigenvalues and eigensections that appear in the assertion of the theorem. Using the natural identification
	\begin{align*}
		C^{\infty}(S^2\ol{M}|_{\SetDefine{1}{}\times \wh{M}})=C^{\infty}(\wh{M})\oplus C^{\infty}(T^*\wh{M})\oplus C^{\infty}(S^2\wh{M})
	\end{align*}
	together with the decompositions \eqref{eq:decomp_1forms} and \eqref{eq:decomp_2tensors}, we get that the $L^2$-span of the eigensections we have constructed equals all of $L^2(S^2\ol{M}|_{\SetDefine{1}{}\times \wh{M}})$.
\end{proof}
\begin{cor}\label{cor:decay_kernel_LL}
	Let  $\ol{h}\in \mathrm{ker}(\ol{\Delta}_{\ol{L}})$, not necessarily defined on all of $\ol{M}$. Let $\nu_i$, $i\in \N$ be the eigenvalues of~$\Box_L$ and define
	\begin{align*}
		\xi^L_+:=\min\SetDefine{\mathrm{Re}(\xi_{\pm}(\nu_i))}{i\in\N}\cap (0,\infty),\qquad 	\xi^L_-:=\min\SetDefine{-\mathrm{Re}(\xi_{\pm}(\nu_i))}{i\in\N}\cap (0,\infty).
	\end{align*}
	Then we have:
	\begin{itemize}
		\item[(i)] If $\ol{h}$ is defined on $(0,\epsilon)\times \wh{M}$ and $|\ol{h}|\to 0$ as $r\to 0$, then $\ol{h}=\O_{\infty}(r^{\xi_+})$ as $r\to0$.
		\item[(ii)]	If $\ol{h}$ is defined on $(R,\infty)\times \wh{M}$ and $|\ol{h}|\to 0$ as $r\to\infty $, then 
		we divide into two subcases:
		\begin{itemize}
			\item[(iia)] If $\NotResonanceDominated{\Box_L}$, then $\ol{h}=\O_{\infty}(r^{-\xi_-})$  as $r\to\infty $.
			\item[(iib)] If $\ResonanceDominated{\Box_L}$, then $\ol{h}=\O_{\infty}(r^{-\frac{n-2}{2}}\log(r))$ as $r\to\infty $.
		\end{itemize}
	\end{itemize}
\end{cor}

\subsection{A simple proof of linear stability}
Recall that the Hardy inequality states that
\begin{align*}
	\inf_{\varphi\in C^{\infty}_{\text{cs}}((0,\infty))}\frac{\int_{0}^{\infty} (\partial_r\varphi)^2r^{n-1}ds}{\int_{0}^{\infty} \varphi^2r^{n-3}ds}=\frac{(n-2)^2}{4}.
\end{align*}
\begin{lem}\label{lem:nonnegative_conical_operators}
	Let $\ol{\Delta}_{\ol{V}}$ be a conical self-adjoint Laplace type operator acting on sections of a vector bundle $\ol{V}$ over $\ol{M}$. Let $\Box_{\wh{V}}$ be its tangential operator. Then,
	\begin{align*}
		\ol{\Delta}_{\ol{V}}\geq0
	\end{align*}
	if and only if 
	\begin{align*}
		\Box_{\wh{V}}\geq -\frac{(n-2)^2}{4}.	
	\end{align*}
\end{lem}
\begin{proof}
	Let  $\ol{u}\in C^{\infty}(\ol{V})$ be of the form
	for $\ol{u}=\varphi(r)u\in C^{\infty}(\ol{V})$ with some $\varphi\in C^{\infty}((0,\infty))$ and an eigensection $u\in E(\Box_{\wh{V}},\nu)$ normalized such that $\left\|u\right\|_{L^2}=1$.
	By  using \eqref{eq:conical_structure} and integrating by parts in~$r$, we can write 
	\begin{align*}
		(\ol{\Delta}_{\ol{V}}\ol{u}_i,\ol{u}_i)_{L^2(\ol{g})}&=\int_0^{\infty}\int_{\wh{M}} (|\partial_r\varphi|^2 |u|^2+r^{-2}\varphi^2\langle \Box_{\wh{V}} u,u\rangle) r^{n-1}\dv_{\wh{g}}dr\\
		&=\int_0^{\infty}(|\partial_r\varphi|^2 +\nu r^{-2}\varphi^2) r^{n-1} dr.
	\end{align*}
	Because the space of finite linear combinations of such sections is $H^1$-dense in $ C^{\infty}(\ol{E})$, the result now follows from the Hardy inequality.
\end{proof}
Using this result, we can now give a simple proof of Theorem~\ref{thm_linear_stability}. %
\label{proof_linear_stability}
\begin{proof}[Proof of Theorem~\ref{thm_linear_stability}]
	Recall from the beginning of Subsection~\ref{subsec:tang_op_forms}, that $\mu_i\geq (n-2)$ for $i\in\N$. Moreover, by the Lichnerowicz--Obata eigenvalue inequality \cite{Ob62}, we get $\lambda_i\geq (n-1)$ for $i\in\N$. 
		 By definition of the numbers $\mu^{(1)}_{i,\pm}$, $\lambda^{(2)}_{i,\pm}$ in Theorem~\ref{thm_spectrum_tangential_LL}, we obtain
	\begin{align*}
		\mu^{(1)}_{i,\pm}&\geq \eta(\xi_+(n-1)-1)=0>	-\frac{(n-2)^2}{4},\\
		\lambda^{(2)}_{i,\pm}&\geq \eta(\xi_+(n-1)-2)=3-n>-\frac{(n-2)^2}{4}.
	\end{align*}
	Therefore, the assertion directly follows from Lemma~\ref{lem:nonnegative_conical_operators} and Theorem~\ref{thm_spectrum_tangential_LL}.
\end{proof}

\subsection{Decay of perturbations satisfying the Bianchi gauge}
In the previous subsection, we have established possible growth and decay rates for homogeneous solutions of the equation $\ol{\Delta}_L\ol{h}=0$. However, not all of them are relevant for our further considerations. Instead, we will later only need to consider the solutions which additionally satisfy the equation
$B_{\ol{g}}(\ol{h})=(\ol{\delta}+\frac{1}{2}d\circ \ol{\trace})(\ol{h})=0$ because by \eqref{eq:linearized_bianchi} and  \eqref{eq:linearized_Ricci}, this implies
\begin{align*}
	\ddtzero \ric_{\ol{g}+t\ol{h}}=0,\qquad 	\ddtzero V(\ol{g}+t\ol{h},\ol{g})=0,
\end{align*}
so Ricci-flatness and the Bianchi gauge are both preserved at a linear level.
For the second main theorem, it turns out that we can also exclude the solutions which are given by Lie derivatives.
Recall that for each eigenvalue $\nu\in\spectrum(\Box_L)$ with $\nu\neq-\frac{(n-2)^2}{4}$ and each eigensection $h\in E(\Box_L,\nu)$, we have two tensors $\ol{h}_{\pm}:=r^{\xi_{\pm}(\nu)}h\in \kernel(\ol{\Delta}_L)$. If $\nu=-\frac{(n-2)^2}{4}$, we use the notation $\ol{h}_+=r^{-\frac{n-2}{2}}\log(r)h$ and $\ol{h}_-=r^{-\frac{n-2}{2}}h$.
Throughout this subsection, we also keep the notation from Theorem~\ref{thm_spectrum_tangential_LL}.
\begin{prop}\label{prop:Bianchi_indicial_values}
	Let $\nu$ be an eigenvalue of $\Box_L$,  $h\in E(\Box_L,\nu)$ and $\ol{h}_{\pm}$ as above. Then, the equation
	\begin{align}\label{eq:linearized_Bianchi_gauge}
		B_{\ol{g}}(\ol{h}_{\pm})=0
	\end{align}
    is satisfied by the following tensors $\ol{h}_{\pm}$:
    \begin{itemize}
    	\item[(i)] If $\nu=\kappa_i$ for some $i\in\N$, then $\ol{h}_{\pm}$ both satisfy \eqref{eq:linearized_Bianchi_gauge}.
    	\item[(ii)] If $\nu=\mu^{(1)}_{i,+}$ for some $i\in\N$, then $\ol{h}_{+}$ solves \eqref{eq:linearized_Bianchi_gauge} but $\ol{h}_{-}$ does not.
    	\item[(iii)] If $\nu=\mu^{(1)}_{i,-}$ for some $i\in\N_0$, then $\ol{h}_{-}$ solves \eqref{eq:linearized_Bianchi_gauge} but $\ol{h}_{+}$ does not.
    	\item[(iv)] If $\nu=\lambda^{(2)}_{i,+}$ for some $i\in\N$, then $\ol{h}_{+}$ solves \eqref{eq:linearized_Bianchi_gauge} but $\ol{h}_{-}$ does not.
    	\item[(v)] If $\nu=\lambda^{(2)}_{i,-}$ for some $i\in\N$, then $\ol{h}_{-}$ solves \eqref{eq:linearized_Bianchi_gauge} but $\ol{h}_{+}$ does not.  
        \item[(vi)] If $\nu=\lambda_{i}$ for some $i\in\N$, then $\ol{h}_{+}$ 
        solves \eqref{eq:linearized_Bianchi_gauge} if and only if $h$ is of the form
        \begin{align*}
	        h=\mathring{\wh{\nabla}}^2v
	        +\frac{\lambda_i(n-2)}{n(n-1)}v\mathring{\wh{g}}-\frac{1}{2}(n-2)\wh{\nabla} v\odot dr+\frac{1}{n}(\xi_+(\lambda_i)-\xi_-(\lambda_i)-2)\xi_-(\lambda_i)v\cdot \ol{g}
        \end{align*}
        and 
        $\ol{h}_{-}$ 
        solves \eqref{eq:linearized_Bianchi_gauge} if and only if $h$ is of the form
        \begin{align*}
        	h=\mathring{\wh{\nabla}}^2v
        	+\frac{\lambda_i(n-2)}{n(n-1)}v\mathring{\wh{g}}-\frac{1}{2}(n-2)\wh{\nabla} v\odot dr+\frac{1}{n}(\xi_-(\lambda_i)-\xi_+(\lambda_i)-2)\xi_+(\lambda_i)v\cdot \ol{g}
        \end{align*}
        for some function $v\in E(\wh{\Delta},\lambda_i)$.
        \item[(vii)] If $\nu=\lambda_{0}=0$, then $\ol{h}_{+}$ solves \eqref{eq:linearized_Bianchi_gauge} but $\ol{h}_{-}$ does not.
        \item[(viii)] If $\nu=\lambda_{0,-}^{(2)}=2n$, then $\ol{h}_{-}$ solves \eqref{eq:linearized_Bianchi_gauge} but $\ol{h}_{+}$ does not.
	\end{itemize}
	Furthermore, only in the cases (i) and (iv), the tensors $\ol{h}_{\pm}$ which solve \eqref{eq:linearized_Bianchi_gauge} are not given by Lie derivatives of the metric $\ol{g}$ along any vector fields.
\end{prop}
\begin{proof}
	The tensors $\ol{h}_{\pm}$ were constructed in \eqref{eq:h_pm} in the proof of Theorem~\ref{thm_spectrum_tangential_LL} by eigensections $v\in C^{\infty}(\wh{M})$, $\omega\in C^{\infty}(T^*\wh{M})$ and $h\in C^{\infty}(TT(\wh{M}))$. If two different tensors $\ol{h}_{\pm}$, $\ol{k}_{\pm}$
	 are constructed out of such, the 1-forms
	\begin{align*}
		B_{\ol{g}}(\ol{h}_{\pm}),\qquad 	B_{\ol{g}}(\ol{k}_{\pm})
	\end{align*}
	are linearly independent unless these both tensors are constructed from the same eigensection and have the same growth rate. Therefore, we can check the main assertion in each of the cases (i)-(viii) separately.
	Throughout the proof, we use the formulas found in the proof of Theorem~\ref{thm_spectrum_tangential_LL}.
	\\
	\begin{enumerate}
		\item[(i)] If $h\in  E(\Box_L,\kappa_i)$, then $h$ is a TT-tensor on $\wh{M}$ and $\ol{h}_{\pm}$ are both TT-tensors on $\ol{M}$ by Lemma~\ref{ttlemma} and thus solve \eqref{eq:linearized_Bianchi_gauge}.
		\item[(ii)] If $h\in  E(\Box_L,\mu^{(1)}_{+})$, then $h=\wh{\delta}^*\omega+\frac{1}{2}(\xi_{\pm}(\mu_i+1)-1)\omega\odot dr $, where $\omega\in E(\wh{\Delta}_1,\mu)=E(\Box_1,\mu+1)$. Furthermore, we have
		\begin{align*}
			\ol{h}_{+}&=r^{\xi_{+}(\mu)}	(\wh{\delta}^*\omega+\frac{1}{2}(\xi_{+}(\mu+1)-1)\omega\odot dr )=\ol{\delta}^*(r^{\xi_{+}(\mu+1)}\omega).
		\end{align*}
		Because $\wh{\delta}\omega=0$, $\ol{h}_{+}$ is trace free. Because $\omega\in E(\Box_1,\mu+1)$, we have $r^{\xi_{+}(\mu+1)}\omega\in\ker(\ol{\Delta}_1)$. Therefore, 
		\begin{align*}
			\ol{\delta}\ol{h}_{+}=\ol{\delta}(\ol{\delta}^*(r^{\xi_{+}(\mu+1)}\omega))=\ol{\Delta}_1(r^{\xi_{+}(\mu+1)}\omega)=0,
		\end{align*}
		which implies that $\ol{\delta}\ol{h}_{+}$ is a TT-tensor. On the other hand, the tensor
		\begin{align*}
			\ol{\delta}\ol{h}_{-}=r^{\xi_{-}(\mu)}h=r^{2-n-\xi_{+}(\mu)}h=r^{2-n}\ol{\delta}\ol{h}_{+}
		\end{align*}
		is still trace free, but not divergence free, as
		\begin{align*}
			\ol{\delta}\ol{h}_{-}=\ol{\delta}(r^{2-n}\ol{h}_{+})=
			r^{2-n}\ol{\delta}(\ol{h}_{+})+(2-n)r^{1-n}\ol{h}_{+}(dr,.)
			=\frac{2-n}{2}(\xi_{+}(\mu+1)-1)r^{1-n}\omega.
		\end{align*}
		In particular, we have
		\begin{align*}
			(\ol{\delta}+\frac{1}{2}d\circ \ol{\trace})(\ol{h}_{-})=\ol{\delta}(\ol{h}_{-})\neq0.
		\end{align*}
		\item[(iii)] This is completely analogous to (ii), with the positions of $+$ and $-$ interchanged throughout the proof.
		\item[(iv)]  If $h\in  E(\Box_L,\lambda^{(2)}_{+})$, then $h=\mathring{\wh{\nabla}}^2v-\left(\frac{\lambda}{n-1}-\xi_{+}(\lambda)\right) v\mathring{\wh{g}}+(\xi_{+}(\lambda)-1)\wh{\nabla} v\odot dr $, where $v\in E(\wh{\Delta},\lambda)$. Furthermore, we have
		\begin{align*}
			\ol{h}_{+}	&=r^{\xi_{+}(\lambda^{(2)}_{+})}\left(\mathring{\wh{\nabla}}^2v-\left(\frac{\lambda}{n-1}-\xi_{+}(\lambda)\right) v\mathring{\wh{g}}+(\xi_{+}(\lambda)-1)\wh{\nabla} v\odot dr\right)\\
			&=\mathring{\ol{\nabla}}^2(r^{\xi_{+}(\lambda)}v)={\ol{\nabla}}^2(r^{\xi_{+}(\lambda)}v)=\ol{\delta}^*(d(r^{\xi_{+}(\lambda)}v)).
		\end{align*}
		Because $v\in E(\wh{\Delta},\lambda)$, we have $r^{\xi_{+}(\lambda)}v\in\ker(\ol{\Delta})$. Therefore, 
		\begin{align*}
			\ol{\delta}\ol{h}_{+}
			=	\ol{\delta}(\mathring{\ol{\nabla}^2}(r^{\xi_{+}(\lambda)}v))=\left(1-\frac{1}{n}\right)\ol{\Delta}_1(\ol{\nabla}(r^{\xi_{+}(\lambda)}v))=0,
		\end{align*}
		which implies that $\ol{\delta}\ol{h}_{+}$ is a TT-tensor  and thus solves \eqref{eq:linearized_Bianchi_gauge}.
		On the other hand, the tensor
		\begin{align*}
			\ol{\delta}\ol{h}_{-}=r^{\xi_{-}(\lambda^{(2)}_{+})}h=r^{2-n-\xi_{+}(\lambda^{(2)}_{+})}h=r^{2-n}\ol{\delta}\ol{h}_{+}
		\end{align*}
		is still trace free, but not divergence free, as
		\begin{align*}
			\ol{\delta}\ol{h}_{-}&=\ol{\delta}(r^{2-n}\ol{h}_{+})=
			r^{2-n}\ol{\delta}(\ol{h}_{+})+(2-n)r^{1-n}\ol{h}_{+}(dr,.)\\
			&=(2-n)(\lambda-(n-1)\xi_+(\lambda))\lambda vdr+	
			(2-n)(\xi_{+}(\lambda)-1)r^{1-n}\wh{\nabla} v.
		\end{align*}
		In particular, we have
		\begin{align*}
			B_{\ol{g}}(\ol{h}_{-})=\ol{\delta}(\ol{h}_{-})\neq0.
		\end{align*}
		\item[(v)] This is completely analogous to (iv), with the positions of $+$ and $-$ interchanged throughout the proof.
		\item[(vi)] If $h\in  E(\Box_L,\lambda)$, then 
		\begin{align*}
			h=\mathring{\wh{\nabla}}^2v
			+\frac{\lambda(n-2)}{n(n-1)}v\mathring{\wh{g}}-\frac{1}{2}(n-2)\wh{\nabla} v\odot dr+w\cdot \ol{g},
			\end{align*}
		with $v,w \in E(\wh{\Delta},\lambda) $. Let us set $w=0$ for the moment. Then we have
		\begin{align*}
			\ol{h}_{+}&=r^{\xi_+(\lambda)}\left(\mathring{\wh{\nabla}}^2v
			+\frac{\lambda(n-2)}{n(n-1)}v\mathring{\wh{g}}-\frac{1}{2}(n-2)\wh{\nabla} v\odot dr\right)\\
			&=\mathring{\ol{\delta}^*}(r^{\xi_+(\lambda_i)+1}(\xi_-(\lambda)vdr+rdv)).
		\end{align*}
		Observe that $\ol{h}_{+}$ is trace free.
		Let us now abbreviate $\ol{\omega}=r^{\xi_+(\lambda)+1}(\xi_-(\lambda)vdr+rdv)$ and recall that $\ol{\Delta}_1(\ol{\omega})=0$.
		 Then we get
		\begin{align*}
			B_{\ol{g}}\ol{\omega}=\ol{\delta}\mathring{\ol{\delta}^*}\ol{\omega}
			=\frac{1}{2}\ol{\Delta}_1\ol{\omega}+\left(\frac{1}{2}-\frac{1}{n}\right)\ol{\nabla}\ol{\delta}\ol{\omega}=
			\frac{n-2}{2n}d\ol{\delta}\ol{\omega}.
		\end{align*}
		We furthermore compute
		\begin{align*}
			\ol{\delta}(\ol{\omega})&=\ol{\delta}(r^{\xi_+(\lambda)-\xi_-(\lambda)+2}r^{\xi_-(\lambda)-1}(\xi_-(\lambda)vdr+rdv))	\\
			&=\ol{\delta}(r^{\xi_+(\lambda)-\xi_-(\lambda)+2}d(r^{\xi_-(\lambda)}v))\\
			&=-\langle d(r^{\xi_+(\lambda)-\xi_-(\lambda)+2}),d(r^{\xi_-(\lambda)}v)\rangle_{\ol{g}}+r^{\xi_+(\lambda)-\xi_-(\lambda)+2}\ol{\Delta}(r^{\xi_-(\lambda)}v)\\
			&=(\xi_-(\lambda)-\xi_+(\lambda)+2)\xi_-(\lambda)r^{\xi_+(\lambda)}v
		\end{align*}
		and consequently
		\begin{align*}
			B_{\ol{g}}(\ol{h}_{+})&=\frac{n-2}{2n}d\ol{\delta}(\ol{\omega})=\frac{n-2}{2n}(\xi_-(\lambda)-\xi_+(\lambda)+2)\xi_-(\lambda)d(r^{\xi_+(\lambda)}v).
		\end{align*}
		On the other hand, if $v=0$, $\ol{h}_{+}=(r^{\xi_+(\lambda)}w)\cdot \ol{g}$ and
		\begin{align*}
			B_{\ol{g}}(\ol{h}_{+})&=\left(-1+\frac{n}{2}\right)d(r^{\xi_+(\lambda)}w)=\frac{n-2}{2}d(r^{\xi_+(\lambda)}w).	
		\end{align*}
		Now we assume that both $v$ and $w$ are nonzero. From summing up the two subcases we considered before, we see that 
		\begin{align*}
			B_{\ol{g}}(\ol{h}_{+})=0	
		\end{align*}
		if and only if
		\begin{align*}
			nw=(\xi_+(\lambda)-\xi_-(\lambda)-2)\xi_-(\lambda)v
		\end{align*}
		The discussion for $\ol{h}_{-}$ is completely analogous, with interchanging the positions of $+$ and $-$ everywhere in the calculation.
		\item[(vii)] If $h\in  E(\Box_L,\lambda_0=0)$, then $\ol{h}_{+}=\alpha\cdot \ol{g}$ and $\ol{h}_{-}=\alpha r^{2-n}\ol{g}$, where $\alpha\in\R$. In the first case, $\ol{h}_{+}$ is parallel and hence solves \eqref{eq:linearized_Bianchi_gauge} whereas in the second case
		\begin{align*}
			B_{\ol{g}}(\ol{h}_{-})=\alpha\frac{n-2}{2}	d(r^{2-n})=\frac{\alpha}{2}(n-2)^2r^{1-n}dr\neq0.
		\end{align*}
		\item[(viii)] If $h\in  E(\Box_L,\lambda_{0_-}=2n)$, then $h=\alpha\cdot (n-2)\cdot ((n-1)dr\otimes dr-\wh{g})$ for some $\alpha\in\R$ and
		\begin{align*}
			\ol{h}_{-}=\alpha\mathring{\ol{\delta}^*}((2-n)r^{1-n}dr)=\alpha \mathring{\ol{\nabla}}^2r^{2-n}=\alpha {\ol{\nabla}}^2r^{2-n}=\alpha {\ol{\delta}}^*(d(r^{2-n})).
		\end{align*}
		Because $r^{2-n}$ is a harmonic function, we easily conclude
		\begin{align*}
			\ol{\delta}( \mathring{\ol{\nabla}}^2r^{2-n})=\ol{\delta}({\ol{\nabla}}^2r^{2-n})=\ol{\Delta}_1 d(r^{2-n})=0
		\end{align*}
		and so, $\ol{h}_{-}$ solves \eqref{eq:linearized_Bianchi_gauge} because it is a TT-tensor. On the other hand, $\ol{h}_{+}$ is trace free but not divergence free, since
		\begin{align*}
			\ol{h}_{+}=r^{n+2}	\ol{h}_{-},
		\end{align*}
		which implies 
		\begin{align*}
			\ol{\delta}(\ol{h}_{+})=\ol{\delta}(r^{n+2}	\ol{h}_{-})=(n+2)r^{n+1}	\ol{h}_{-}(dr)+r^{n+2}\ol{\delta}\ol{h}_{-}=\alpha(n+2)(n-2)(n-1)rdr\neq 0.
		\end{align*}
	\end{enumerate}
	To finish the proof, we recall that the cases (ii)-(v) and (viii), the tensors $\ol{h}_{\pm}$ which are in the kernel of $B_{\ol{g}}$ are of the form $\ol{h}_{\pm}=\ol{\delta}^*(\ol{\omega}_{\pm})=\frac{1}{2}\mathcal{L}_{(\ol{\omega}_{\pm})^{\sharp}}\ol{g}$. This is also the case for the metric (case (vii)), since $\ol{g}=\ol{\delta}^*(rdr)=\frac{1}{2}\mathcal{L}_{r\partial_r}\ol{g}$.
\end{proof}
\noindent Let us label the indicial set of $\ol{\Delta}_L$ by
\begin{align*}
	E_L&:=\SetDefine{\xi_{\pm}(\kappa_i),\xi_{\pm}(\mu_i+1)- 1,\xi_{\pm}(\mu_i+1)+ 1
	,\xi_{\pm}(\lambda_i)-2,\xi_{\pm}(\lambda_i),\xi_{\pm}(\lambda_i)+2
	}{i\in\N}\\
	&\qquad\cup\SetDefine{-n,2-n,0,2}{}
\end{align*}
and recall from Corollary~\ref{cor:decay_kernel_LL} the notation
\begin{align*}
	\xi^L_+:=\min \mathrm{Re}(E_L)\cap(0,\infty),\qquad \xi^L_-:=\min \mathrm{Re}(-E_L)\cap(0,\infty).
\end{align*}
We denote the indicial set of $\ol{\Delta}_L$ on tensors satisfying the linearized Bianchi gauge by
\begin{align*}
	E_B=\SetDefine{\xi_{\pm}(\kappa_i),\xi_{\pm}(\mu_i+1)- 1
	,\xi_{\pm}(\lambda_i)-2,\xi_{\pm}(\lambda_i)
	}{i\in\N}\cup\SetDefine{-n,0}{}
\end{align*}
and set
\begin{align*}
	\xi^B_+:=\min \mathrm{Re}(E_B)\cap(0,\infty),\qquad \xi^B_-:=\min \mathrm{Re}(-E_B)\cap(0,\infty).
\end{align*}
As in the introduction, we label the subset of $E_B$ not corresponding to Lie derivatives by
\begin{align*}
	E:=\SetDefine{\xi_{\pm}(\kappa_i),\xi_{\pm}(\lambda_i)
	}{i\in\N}\cup\SetDefine{0}{}.
\end{align*}
and we set
\begin{align*}
	\xi_+:=\min \mathrm{Re}(E)\cap(0,\infty),\qquad \xi_-:=\min \mathrm{Re}(-E)\cap(0,\infty).
\end{align*}
Obviously, we have
\begin{align*}
	\xi_+\geq 	\xi^B_+\geq \xi^L_+,\qquad 	\xi_-\geq 	\xi^B_-\geq \xi^L_-.
\end{align*}
\begin{cor}\label{cor:decay_gauged_kernel_LL} 
	Let
	$\ol{h}\in \mathrm{ker}(\ol{\Delta}_{L})\cap \kernel(B_{\ol{g}})$, not necessarily defined on all of $\ol{M}$. Then we have the following:
	\begin{itemize}
		\item[(i)] If $\ol{h}$ is defined on $(0,\epsilon)\times \wh{M}$ for some $\epsilon>0$ and $|\ol{h}|\to 0$ as $r\to 0$, then $\ol{h}=\O_{\infty}(r^{\xi^B_+})$ as $r\to0$.
		\item[(ii)]	If $\ol{h}$ is defined on $(R,\infty)\times \wh{M}$ for some $R>0$ and $|\ol{h}|\to 0$ as $r\to\infty $, then 
		we divide into two subcases:
		\begin{itemize}
			\item[(iia)] If $\NotResonanceDominated{\Box_L}$, then $\ol{h}=\O_{\infty}(r^{-\xi^B_-})$  as $r\to\infty $.
			\item[(iib)] If $\ResonanceDominated{\Box_L}$, then $\ol{h}=\O_{\infty}(r^{-\frac{n-2}{2}}\log(r))$  as $r\to\infty $.
		\end{itemize}
	\end{itemize}
\end{cor}

\section{Optimal coordinates for conifolds}\label{sec:coordinates}

\subsection{Analysis on conifolds}\label{subsec:analysis_conifolds}
In order to do analysis of partial differential equations on conifolds, one has to work with weighted function spaces. In the following we recall some well-known facts for weighted Sobolev and H\"{o}lder spaces and Laplace type operators on conifolds which can be found in many articles and textbooks, see e.g.\ \cite{Bar86,LM85,Mel93,Pac13}.
\begin{defn}
	Let $(M,g)$ be a conifold. A  smooth function  $\rho:M\to(0,\infty)$ is called a \emph{radius function} if at each end $M_i$, there exists an asymptotic chart $\varphi_i$ such that $(\varphi_i)_*\rho=r$.
\end{defn}
For simplicity, we assume that $M$ has only one end $M\setminus K$ which is either asymptotically conical or conically singular.
Moreover, for fixed $R\in\R$, we will use the notation 
\begin{equation}\label{eq:sublevel-set-of-the-radius-function}
	M_{<R}:=\SetDefine{p\in M}{\rho(p)<R},
\end{equation}
and similarly with $M_{>R}:=\SetDefine{p\in M}{\rho(p)>R}$.
On a cone $\ol{M}$, we pick the function $r$ as radius function and use the notation $\ol{M}_{<R}$ etc. accordingly (cf. Remark~\ref{rem:multiple ends}.)

For $p\in[1,\infty)$ and a weight $\beta\in \mathbb{R}$, we define the space  $L^2_{\beta}(M)$ as the closure of $C^{\infty}_{\text{cs}}(M)$ with respect to the norm
\begin{align*}
	\left\|u\right\|_{L^2_{\beta}}=\left(\int_M |\rho^{-\beta}u|^2\rho^{-n}d\mu\right)^{1/2},
\end{align*}
and the weighted Sobolev spaces $H^{k}_{\beta}(M)$ as the closure of $C^{\infty}_{\text{cs}}(M)$ under
\begin{align*}
	\left\|u\right\|_{H^{k}_{\beta}}=\sum_{l=0}^k\left\|\nabla^lu\right\|_{L^{2}_{\beta-l}}.
\end{align*}
The weighted H\"{o}lder spaces are defined as the set of maps $u\in C^{k,\alpha}_{\text{loc}}(M)$, $\alpha\in(0,1)$ such that the norm
\begin{align*}
	\left\|u\right\|_{C^{k,\alpha}_{\beta}}=&\sum_{l=0}^k\sup_{x\in M} \rho^{-\beta+l}(x)|\nabla^lu(x)|\\&\quad+\sup_{\substack{x, y\in M\\ 0<d(x,y)<\mathrm{inj}(M)}}\min\SetDefine{\rho^{-\beta+k+\alpha}(x),\rho^{-\beta+k+\alpha}(y)}{}\frac{|\tau^y_x\nabla^ku(x)-\nabla^k u(y)|}{d(x,y)^\alpha},
\end{align*}
is finite. Here $\tau_x^y$ denotes the parallel transport from $x$ to $y$ along the shortest geodesic joining $x$ and $y$, and $d(x,y)$ is the Riemannian distance between $x$ and $y$. All these spaces are Banach spaces, the spaces $H^k_{\beta}(M)$ are Hilbert spaces and their underlying topological vector space structures do not depend on the choice of the radius function $\rho$. All these definitions extend to Riemannian vector bundles with a metric connection in an obvious manner \cite{Pac13}.
In the literature, there are different notational conventions for weighted spaces. We follow the more standard convention used in \cite{Bar86,Pac13}.
We have, for every $\beta\in\R$, $k,l\in\N_0$ and $\alpha\in(0,1)$, the obvious embeddings 
\begin{align*}
	H^{k+l}_{\beta}\subset	H^k_{\beta},\qquad C^{k+1,\alpha}_{\beta}\subset C^{k,\alpha}_{\beta}.
\end{align*}
In the asymptotically conical case, we further have the embeddings
\begin{align*}
	H^k_{\beta}\subset H^k_{\beta'},\qquad  C^{k,\alpha}_{\beta}\subset C^{k,\alpha}_{\beta'},\qquad C^{k,\alpha}_{\beta}\subset H^k_{\beta'}	
\end{align*}
for any $\beta<\beta'$, the Sobolev embedding
\begin{align*}
	H^{k+l}_{\beta}\subset	C^{k,\alpha}_{\beta'}
\end{align*}
for $l>n/2$ and  $\beta<\beta'$,
and pointwise multiplication extends to a continuous map
\begin{align*}
	\cdot:H^k_{\beta}	\times H^k_{\beta}\to H^k_{\beta}
\end{align*}
for $\beta\leq0$.
 Conversely, in the conically singular case, we have 
\begin{align*}
	H^k_{\beta}\subset H^k_{\beta'},\qquad  C^{k,\alpha}_{\beta}\subset C^{k,\alpha}_{\beta'},\qquad C^{k,\alpha}_{\beta}\subset H^k_{\beta'}	
\end{align*}
for any $\beta'<\beta$,
the Sobolev embedding
\begin{align*}
	H^{k+l}_{\beta}\subset	C^{k,\alpha}_{\beta'}
\end{align*}
for $l>n/2$ and  $\beta'<\beta$,
and pointwise multiplication extends to a continuous map
\begin{align*}
	\cdot:H^k_{\beta}	\times H^k_{\beta}\to H^k_{\beta}
\end{align*}
for $\beta\geq0$.
\begin{rem}\label{rem:multiple ends}
	If $M$ is a conifold with multiple ends which might be AC as well as CS, these function spaces are generalized as follows: We index the  AC (\enquote{large}) ends by $\SetDefine{1,\ldots,a}{}$ and the CS (\enquote{small}) ends by $\SetDefine{1,\ldots,b}{}$. Consider the vector $\beta=(\zeta,\sigma)=(\zeta_1,\ldots,\zeta_a,\sigma_1,\ldots,\sigma_b)\in  \R^{a+b}$, where each $\zeta_i$ refers to an AC end and each $\sigma_j$ refers to a CS end. We write $\zeta\leq\zeta'$ (resp. $ \geq,<,> $) whenever $\zeta_i\leq \zeta'_i$ (resp. $ \geq,<,> $) for each $i$, and similarly for $\sigma$. For $i\in\Z$, we write $\beta+i:=(\zeta+i,\sigma+i):=(\zeta_1+i,\ldots,\zeta_a+i,\sigma_1+i,\ldots,\sigma_b+i)$. Given $\beta=(\zeta,\sigma)\in \R^{a+b}$, choose a smooth function on $M$ (again denoted by $\beta:M\to\R$) which at each AC end $M_i$ coincides with $\zeta_i$ and at each CS end $M_j$ coincides with $\sigma_j$. Then we can define the weighted Sobolev spaces $H^k_{\beta}=H^k_{\zeta,\sigma}$ and the weighted H\"{o}lder spaces $C^{k,\alpha}_{\beta}=C^{k,\alpha}_{\zeta,\sigma}$ exactly as above. The embedding properties from above hold for $\beta=(\zeta,\sigma),\beta'=(\zeta',\sigma')\in \R^{a+b}$ with $\zeta<\zeta'$ and $\sigma'<\sigma$ and the multiplication properties for $\zeta\leq 0$ and $\sigma\geq0$.
\end{rem}
\begin{defn}
	A Laplace type operator $\Delta_V$ on a Riemannian vector bundle $V$ with connection over an asymptotically conical/conically singular manifold $M$ is called \emph{asymptotically conical/conically singular} of order $\alpha$ if there exists a Riemannian vector bundle $\ol{V}$ with connection over $(R,\infty)\times \wh{M}$ (resp.\ $(0,\epsilon)\times \wh{M}$), a conical operator $\ol{\Delta}_{\ol{V}}$, a vector bundle isomorphism $\Phi:V|_{M\setminus K}\to \ol{V}$ covering an asymptotic chart $\phi$ and a constant $\alpha>0$, such that for any section $u=\mathcal{O}_{\infty}(r^{\beta})$, $\beta\in\R$, we have
	\begin{align*}
		\ol{\nabla}^k (\ol{\Delta}_{\ol{V}}\Phi(u)-\Phi(\Delta_V u))=\mathcal{O}_{\infty}(r^{\beta-\alpha-2-k})\qquad\text{for all }k\in\N_0
	\end{align*}
	as $r\to\infty$ in the AC case, and
	\begin{align*}
		\ol{\nabla}^k (\ol{\Delta}_{\ol{V}}\Phi(u)-\Phi(\Delta_V u))=\mathcal{O}_{\infty}(r^{\beta+\alpha-2-k})\qquad\text{for all }k\in\N_0
	\end{align*}
	as $r\to0$ in the CS case. A Laplace type operator over a conifold is called a \emph{conifold operator} if it is  asymptotically conical and conically singular at each asymptotically conical and conically singular end, respectively.
\end{defn}
\noindent For any $\beta\in\R$, a conifold operator $\Delta_V$ defines a continuous map
\begin{align}\label{eq:Laplacian_weighted_Sobolev}
	\Delta_V:W^{k,p}_\beta(V)\to W^{k-2,p}_{\beta-2}(V).
\end{align} 
Additionally, it is Fredholm for all $\beta\in\R$ up to a discrete set, see e.g.\ \cite[Proposition~5.64]{Mel93}: Let
\begin{align*}
	D=\SetDefine{\mathrm{Re}(\xi_{\pm}(\nu))}{\nu\in \spectrum(\Box_V)},
\end{align*} 
where $\Box_V$ is the tangential operator of the conical operator $\ol{\Delta}_{\ol{V}}$. In other words $D$ is the set of the real part of the indicial roots of $\ol{\Delta}_{\ol{V}}$.
We call $\beta$ \emph{nonexceptional} whenever $\beta\in\R\setminus D$ and \emph{exceptional}, whenever $\beta\in D$.
It is a standard fact that the operator \eqref{eq:Laplacian_weighted_Sobolev} is Fredholm if and only if $\beta$ is nonexceptional. An analogous statement holds for weighted H\"{o}lder spaces but we will not need it for our purposes.

To conclude this short introduction, we mention an important duality argument.
Note that the $L^2$-scalar product induces a bounded map
\begin{align*}
	(.,.): L^2_{\beta}\times L^2_{-n-\beta}\to\R
\end{align*}
consequently we can think of $L^2_{-n-\beta}$ as being the dual space of $L^2_{\beta}$. Motivated by this, we define
\begin{align*}
	H^{-k}_{\beta}:=(H^{k}_{-n-\beta})',\qquad k\in\N.
\end{align*}
If $\Delta_V$ is formally self-adjoint, the dual map of
\begin{align*}
		\Delta_V:H^{k}_\beta(V)\to W^{k-2}_{\beta-2}(V)
\end{align*}
is, via the $L^2$-pairing, again given by $\Delta_V$,  seen as a map
\begin{align*}
	\Delta_V:H^{2-k}_{2-n-\beta}(V)\to H^{-k}_{-n-\beta}(V).
\end{align*}
Recall that a Fredholm map is surjective if its dual map is injective. Therefore,
\begin{align*}
	\Delta_V:H^{k}_\beta(V)\to H^{k-2}_{\beta-2}(V)
\end{align*}
is surjective if and only if 
\begin{align*}
	\mathrm{ker}_{H^{2-k,q}_{2-n-\beta}}(\Delta_V)=0.
\end{align*}
However, by elliptic regularity for weighted spaces, 
\begin{align*}
	\mathrm{ker}_{H^{2-k}_{2-n-\beta}}(\Delta_V)=\mathrm{ker}_{H^{l}_{2-n-\beta}}(\Delta_V)
\end{align*}
for every $l\in\N$. This gives us a useful criterion for checking isomorphism properties for $\Delta_V$.

For example, consider the conical operator $\ol{\Delta}_{\ol{V}}$ over a Riemannian cone $(\ol{M},\ol{g})$. 
In view of Remark~\ref{rem:multiple ends}, we may think of $\ol{M}$ as a conifold with one AC and one CS end and choose the same weight $\beta$ on both ends.
We may choose $\rho=r$ as the radius function.
By \eqref{eq:harmonic_expansion}, an element $u\in\kernel(\ol{\Delta}_{\ol{V}})$ defined on all of $\ol{M}$ can not satisfy $u\in o(r^{\alpha})$ both as $r\to 0$ and $r\to \infty$ for any $\alpha\in \R$. For this reason,
\begin{align*}
	\ol{\Delta}_{\ol{V}}:H^{k}_\beta(\ol{V})\to H^{k-2}_{\beta-2}(\ol{V})
\end{align*}
is injective for all weights $\beta$. By the duality argument above, $\ol{\Delta}_{\ol{V}}$ is also surjective for all nonexceptional weights $\beta$.
Summing up, we conclude:
\begin{prop}\label{prop:conical_isomorphism}
	A conical operator $\ol{\Delta}_{\ol{V}}$ over a Riemannian cone $(\ol{M},\ol{g})$, seen as an operator
	\begin{align*}
		\ol{\Delta}_{\ol{V}}:H^{k}_\beta(V)\to H^{k-2}_{\beta-2}(V)
	\end{align*}	
	 is an isomorphism for every nonexceptional weight $\beta$.
\end{prop}

\subsection{Decay of Ricci-flat metrics on the cone}\label{subsec:decay_gauged_metrics}
Recall that for  two different Riemannian metrics $g,\tilde{g}$ the vector field $V({g},\tilde{g})$ is given in local coordinates by
\begin{align*}
	V(g,\tilde{g})^l:=g^{ij}(\Gamma(g)_{ij}^l-\Gamma(\tilde{g})_{ij}^l).	
	\end{align*}
\begin{defn}\label{def:Bianchi_gauge}
	We say that a metric $g$ is in \emph{Bianchi gauge} with respect to $\tilde{g}$ if $V(g,\tilde{g})=0$.
\end{defn}
\begin{thm}\label{thm:Ricci_flat_bianchi_gauge}
	Let $(\ol{M},\ol{g})$ be a Ricci-flat cone and let $g$ be a Ricci-flat metric defined on an open set $U\subset \ol{M}$ which is in Bianchi gauge with respect to $\ol{g}$.
	Then the following assertions hold: 
	\begin{itemize}
		\item[(i)] If $U$ is an open neighborhood of $0$ (i.e. if $U$ contains $\ol{M}_{<R}$ for some $R\in\R$) and $g-\ol{g}= \mathcal{O}_2(r^{\alpha})$ for some $\alpha>0$ as $r\to0$, then we  have $g-\ol{g}= \mathcal{O}_{\infty}(r^{\xi^B_+})$ as $r\to0$. 
		\item[(ii)]	If $U$ is an open neighborhood of $\infty$ (i.e. if $M\setminus U$ is compact) and $g-\ol{g}\in \mathcal{O}_2(r^{-\alpha})$ for some $\alpha>0$ as $r\to\infty$,
		 then 
		we divide into two subcases:
		\begin{itemize}
			\item[(iia)] If $\NotResonanceDominated{\Box_L}$, then $g-\ol{g}=\O_{\infty}(r^{-\xi^B_-})$  as $r\to\infty $.
			\item[(iib)] If $\ResonanceDominated{\Box_L}$, then $g-\ol{g}=\O_{\infty}(r^{-\frac{n-2}{2}}\log(r))$  as $r\to\infty $.
		\end{itemize}
	\end{itemize}
\end{thm}
\begin{proof}
	In this proof, all norms and tensor products are taken with respect to $\ol{g}$.  By the assumptions of the theorem, we have the equation
	\begin{align*}
		2\ric_g=\mathcal{L}_{V(g,\ol{g})}\ol{g}
	\end{align*}
	on $U$, which can be written with respect to the difference  $h=g-\ol{g}$ as
	\begin{align}\label{eq:ricci_de_Turck2}
		\ol{\Delta}_Lh=g^{-1}*\ol{\mathrm{Rm}}*h*h+g^{-1}*g^{-1}*\ol{\nabla} h*\ol{\nabla} h+g^{-1}*\ol{\nabla}^2h* h,
	\end{align}
	where, as usual, $*$ denotes finite linear combinations of tensor contractions with covariantly constant coefficients.
	This follows essentially from \cite[Lemma~2.1]{Shi89}, but is carried out in more detail in \cite[Lemma~3.1]{KP20}.
	From here on, the proof is a standard iteration procedure in weighted function spaces, but we decided to present it here for completeness. We focus on case (i), the other one is completely analogous.
	Without loss of generality, we assume that $h$ is defined on $\ol{M}_{< R}$ for some $R>0$.	
	We extend $h$ smoothly to a tensor $\ol{h}$ on $\ol{M}$ such that
	\[
	\ol{h}\equiv0 \text{ on }\ol{M}_{>2R},\qquad \ol{h}\equiv h\text{ on }\ol{M}_{< R}.
	\]
	At first, by elliptic regularity, $\ol{h}=\mathcal{O}_{\infty}(r^{\alpha})$ as $r\to0$ (and for trivial reasons also as $r\to\infty$).
    Thus  by \eqref{eq:ricci_de_Turck2},    	
	\begin{equation}
		 \ol{\Delta}_L \ol{h} = \mathcal{O}_{\infty}(r^{2\alpha-2}),
	\end{equation}
    in both the cases $r\to\infty$ and $r\to 0$. Therefore,
    \begin{align*}
    	 \ol{\Delta}_L \ol{h}\in H^k_{2\alpha-2-\epsilon}(S^2M)
    \end{align*}
	for any $k\in \N_0$ and for any $\epsilon>0$
    and by Proposition~\ref{prop:conical_isomorphism}, we find  a tensor $h_1\in H^{k+2}_{-2\alpha-\epsilon}(S^2M)$ such that
    \begin{align*}
    	\ol{\Delta}_Lh_1=\ol{\Delta}_L \ol{h},
    \end{align*}
	provided that $2\alpha-\epsilon$ is a nonexceptional weight.
	Note that $h_1$ is independent of the choice of $k$ (but not of the choice of $\epsilon$) as $H^{k+2}_{2\alpha-\epsilon}\subset H^{k'+2}_{2\alpha+\epsilon}$ if  $k\leq k'$. In particular, $h_1\in H^{k}_{2\alpha-\epsilon}$ for all $k\in \N$ and $\epsilon>0$ so by Sobolev embedding, $h_1\in \mathcal{O}_{\infty}(r^{2\alpha-2\epsilon})$.
      
    Because $h_0:=\bar{h}-h_1\in\ker(\Delta_L)$ and $h_0\in \mathcal{O}_{\infty}(r^{\beta})$ as $r\to 0$, we know by  Corollary~\ref{cor:decay_harmonic_sections} that $h_1\in  \mathcal{O}_{\infty}(r^{\xi_+})$. We obtain on the set $\ol{M}_{< R}$ that
    \begin{align*}
    	h=\ol{h}=h_0+h_1\in \mathcal{O}_{\infty}(r^{\xi_+})+\mathcal{O}_{\infty}(r^{2\alpha-2\epsilon}).
    \end{align*}
	as $r\to 0$. If $2\alpha-2\epsilon<\xi_+$ for some $\epsilon>0$, we are done. Otherwise we repeat the same procedure again, starting with the new decay rate $\alpha':=2\alpha-2\epsilon$. After iterating this procedure at most a finite number of times, we will be in the situation where $2\alpha-2\epsilon<\xi_+$. In this case, we can now conclude
	\begin{align*}
		h=\ol{h}=h_0+h_1\in \mathcal{O}_{\infty}(r^{\xi_+}),
	\end{align*}
	where $\ol{\Delta}_Lh_0=0$ and $h_1\in \mathcal{O}_{\infty}(r^{\xi_++\epsilon})$ for some $\epsilon>0$.
	The equation $V(g,\ol{g})=0$ is equivalent to
	\begin{align*}
		0=\frac{1}{2}g^{ij}(\ol{\nabla}_ih_{jk}+\ol{\nabla}_{j}h_{ik}-\ol{\nabla}_kh_{ij})
		&=B_{\ol{g}}(h)+\frac{1}{2}(g^{ij}-\ol{g}^{ij})(\ol{\nabla}_ih_{jk}+\ol{\nabla}_{j}h_{ik}-\ol{\nabla}_kh_{ij})\\
		&=B_{\ol{g}}(h_0)+B_{\ol{g}}(h_1)-h_{lm}{g}^{il}\ol{g}^{jm}(\ol{\nabla}_ih_{jk}+\ol{\nabla}_{j}h_{ik}-\ol{\nabla}_kh_{ij})\\
		&=B_{\ol{g}}(h_0)+\mathcal{O}(r^{\xi_+-1+\epsilon}),
	\end{align*}
	and because $B_{\ol{g}}(h_0)\in \mathcal{O}_{\infty}(r^{\xi_+-1})$ we conclude $B_{\ol{g}}(h_0)=0$. Thus, if $\xi_+<\xi^B_+$, we get $h_0=0$ and therefore $h\in \mathcal{O}_{\infty}(r^{\xi_++\epsilon})$. We then continue with the same procedure as above till we arrive at a decomposition $h=h_0+h_1$  with $\Delta_Lh_0=0$ and $B_{\ol{g}}(h_0)=0$, where $h_1$ decays faster than $h_0$. By Corollary~\ref{cor:decay_gauged_kernel_LL},
	we know that $h_0=\mathcal{O}_{\infty}(r^{\xi^B_+})$ which
	implies the desired result.
\end{proof}

\subsection{The Bianchi gauge}
In this subsection, we want to find out under which conditions the Bianchi gauge condition actually defines a reasonable gauge. For simplicity, we assume again that our manifold has only one end which is either asymptotically conical or conically singular.
 We start with the following observation:
\begin{lem}
	The set
	\begin{align*}
		\mathcal{G}_{\beta}^k:=\SetDefine{\tilde{g}\in H^k_{\beta}(S^2_+M)}{V(g,\tilde{g})=0},
	\end{align*}
	is an open subset of a vector space. In particular, it is a submanifold of $H^k_{\beta}(S^2_+M)$.
\end{lem}
\begin{proof}
	Choose a point $p\in M$, and $h=\tilde{g}-g$. In $g$-normal coordinates around $p$, we compute at $p$ that
	\begin{align*}
		V(g,\tilde{g})^k=	-{g}^{ij}\Gamma(\tilde{g})_{ij}^k
		&=-\frac{1}{2}{g}^{ij}\tilde{g}^{kl}(\partial_i\tilde{g}_{lj}+\partial_j\tilde{g}_{li}-\partial_l\tilde{g}_{ij})\\
		&=-\frac{1}{2}{g}^{ij}\tilde{g}^{kl}(\partial_ih_{lj}+\partial_jh_{li}-\partial_lh_{ij})\\
		&=-\frac{1}{2}{g}^{ij}\tilde{g}^{kl}({\nabla}_ih_{lj}+{\nabla}_jh_{li}-{\nabla}_lh_{ij})
		=\tilde{g}^{kl}B_g(h)_l.
	\end{align*}
	Therefore,
	\begin{align*}
		\mathcal{G}_{\beta}^k=\SetDefine{g+h\in H^k_{\beta}(S^2_+M)}{B_{g}(h)=0}=\SetDefine{h\in H^k_{\beta}(S^2_+M)}{B_{g}(h)=0}
	\end{align*}
	since $B_g(g)=0$.
	Hence, we have redefined $\mathcal{G}_{\beta}^k$ by a linear equation, which proves the lemma.
\end{proof}
\begin{rem}
	The fact that the above gauge condition is linear is the reason why we prefer it over the condition $V(\tilde{g},g)=0$ for fixed $g$. In fact, if $g$ has nontrivial Killing fields, the differential of the linearization is not surjective and the set of metrics satisfying this condition may fail to form a manifold.
\end{rem}

\begin{lem}\label{prop:surj_Laplace}
	The connection Laplacian
\begin{align*}
	\Delta_1: H^{k+2}_{\beta}(T^*M)\to H^{k}_{\beta-2}(T^*M)	
\end{align*}	
is injective for nonexceptional $\beta<0$ and surjective for every nonexceptional $\beta>2-n$.
\end{lem}
\begin{proof}
	If $\omega\in H^{k+2}_{\beta}(TM)$ satisfies $\Delta_1 \omega=0$, then
	\begin{align*}
		\Delta |\omega|^2+2|\nabla \omega|^2=0.
	\end{align*}
	By the maximum principle, we conclude that any bounded harmonic vector field on $M$ is parallel. In particular, its pointwise norm is constant and 
	\begin{align*}
		\Delta_1: H^{k+2}_{\beta}(T^*M)\to H^{k}_{\beta-2}(T^*M)	
	\end{align*}	
is an injective Fredholm operator for every nonexceptional $\beta<0$. The surjectivity for $\beta>2-n$ follows from duality via the $L^2$-pairing, as explained in Subsection~\ref{subsec:analysis_conifolds}.
\end{proof}
\begin{lem}\label{lem:solv_criterion}
	Pick $\beta>0$ such that the two operators
	\begin{align*}
		B: H^k_{\beta}(S^2M)\to H^{k-1}_{\beta-1}(T^*M)
	\end{align*}
	and
	\begin{align*}
		\Delta_1: H^{k+1}_{\beta+1}(T^*M)\to H^{k-1}_{\beta-1}(T^*M)	
	\end{align*}
	are both Fredholm.
	Then,  we have $Bh\in \Delta_1(H^{k+1}_{\beta+1}(T^*M))$ if and only if
	\begin{align}\label{eq:orthogonality2}
		(h,B^*\eta)_{L^2}=0
	\end{align}
	for every $\eta \in H^k_{1-n-\beta}(T^*M)$ with $\Delta_1\eta=0$. Here, $B^*\omega=\delta^*\omega+\frac{1}{2}\delta\omega\cdot g$ is the formal adjoint of $B$.
\end{lem}
\begin{proof}
	By duality, $Bh=\Delta_1\omega$ for some $\omega\in H^{k+1}_{\beta+1}(T^*M)$ if and only if 
	\begin{align*}
		(Bh,\eta)_{L^2}=0
	\end{align*}
	for all $\eta \in H^{1-k}_{1-n-\beta}(T^*M)$ with $\Delta_1\eta=0$. The criterion \eqref{eq:orthogonality2} now follows from integration by parts. By elliptic regularity,  $\eta \in H^{k}_{1-n-\beta}(T^*M)$. 
\end{proof}
\begin{lem}\label{lem:compact_perturbation2}
	Let $\beta\in\R$ be as in Lemma~\ref{lem:solv_criterion}. Then
	there exist a number $N\in\N$ and for any open precompact subset $U\subset M$ a set of tensors $h_i$, $i=1,\ldots N$, with support in $\ol{U}$ such that
	\begin{align*}
		H^k_{\beta}(S^2M)=B^{-1}(\Delta_1(H^{k+1}_{\beta+1}(T^*M)))\oplus \SetDefine{\sum_{i=1}^N\alpha_ih_i}{\alpha_i\in \R}.
	\end{align*}
\end{lem}
\begin{proof}
	Note that on a Ricci-flat manifold, $2\delta \circ B^* = \Delta_1$, therefore $\kernel_{H^{k+1}_{\beta+1}}(B^*)\subset \kernel_{H^{k+1}_{\beta+1}}(\Delta_1).$ Since $\kernel_{H^{k+1}_{\beta+1}}\Delta_1$ is finite dimensional, we can find a finite-dimensional vector space $V$ such that
	\begin{align*}
		\kernel_{H^{k+1}_{\beta+1}}(\Delta_1)=\kernel_{H^{k+1}_{\beta+1}}(B^*) \oplus V.
	\end{align*}
	By the commutation formulas \eqref{commutation}, $B^*\eta\in\kernel(\Delta_L)$ for $\eta\in V$. By assumption, $B^*\eta\neq 0$ if $\eta\neq0$. In particular, $B^*\eta$ does not vanish identically on any open subset of $M$ by elliptic theory. For a given open and precompact set $U\subset M$, we choose a bump function $\chi:M\to [0,1]$ which is strictly positive in $U$ and vanishes identically on $M\setminus U$. Then the bilinear form
	\begin{align*}
		A: V\times V\to \R,\qquad (\eta_1,\eta_2)\mapsto \int_M \chi \langle B^*\eta_1,B^*\eta_2\rangle\dv
	\end{align*} 
	is an inner product because the tensors $B^*\eta_i$ do not vanish identically on any open subset of $U$. Therefore, we may choose a basis $\SetDefine{\eta_i}{i=1,\ldots N}$ of $V$ which is orthonormal with respect to the inner product $A$. Define $h_i:=\chi B^*\eta_i$ for $i\in\SetDefine{1,\ldots, N}{}$. 
	Now let  $h\in H^k_{\beta}(S^2M)$ be arbitrary and make the ansatz
	\begin{align*}
		\hat{h}=h-\sum_{i=1}^N\alpha_i h_i,\qquad \alpha_i\in\R.
	\end{align*}
	By Lemma~\ref{lem:solv_criterion}, we have \begin{align}\label{eq:ansatz-for-decomposition}
		\hat{h}\in B^{-1}(\Delta_1(H^{k+1}_{\beta+1}(T^*M)))
	\end{align}
	if and only if 
	$(\hat{h},B^*\eta)_{L^2}=0$ for all $\eta\in V$, or equivalently, if and only if $(\hat{h},B^*\eta_j)_{L^2}=0$ for all $j\in\SetDefine{1,\ldots,N}{}$. By construction of the $h_i$,
	\begin{align*}
		(\hat{h},B^*\eta_j)_{L^2}=(h,B^*\eta_j)_{L^2}-\sum_{i=1}^N\alpha_i(h_i,B^*\eta_j)_{L^2}=(h,B^*\eta_j)_{L^2}-\sum_{i=1}^N\alpha_iA(\eta_i,\eta_j).
	\end{align*}
	Because $A(\eta_i,\eta_j)=\delta_{ij}$ by construction, \eqref{eq:ansatz-for-decomposition} is satisfied if and only if  $\alpha_i=(\hat{h},B^*\eta_i)_{L^2}$ for all $i\in\SetDefine{1,\ldots,N}{}$. This proves the Lemma.
\end{proof}
\begin{lem}\label{lem:complement_gauge}
	We have 
	\begin{align*}
		\kernel_{H^k_{\beta}}(B)\cap \delta^*(H^{k+1}_{\beta+1}(T^*M))=\delta^*(\kernel_{H^{k+1}_{\beta+1}}(\Delta_1)).
	\end{align*}
	Furthermore, for any subspace   $Z^k_{\beta}\subset\kernel_{H^k_{\beta}}(B)$ with
	\begin{align}\label{eq:complement-of-delta-star-kernel}
		\kernel_{H^k_{\beta}}(B)=Z^k_{\beta}\oplus \delta^*(\kernel_{H^{k+1}_{\beta+1}}(\Delta_1)),
	\end{align}
	we also have
	\begin{align*}
		B^{-1}(\Delta_1(H^{k+1}_{\beta+1}(T^*M)))=Z^k_{\beta}\oplus \delta^*(H^{k+1}_{\beta+1}(T^*M)).
	\end{align*} 
\end{lem}
\begin{proof}
	The first assertion is immediate from the formula
	\begin{align*}
		B\circ \delta^*=	\Delta_1+\ric=\Delta_1,
	\end{align*}
	which follows from a straightforward calculation. Since $\kernel_{H^{k+1}_{\beta+1}}(\Delta_1)$ is finite dimensional
	we can choose a subspace  $Z^k_{\beta}\subset\kernel_{H^k_{\beta}}(B)$ such that
	\begin{align*}
		\kernel_{H^k_{\beta}}(B)=Z^k_{\beta}\oplus \delta^*(\kernel_{H^{k+1}_{\beta+1}}(\Delta_1)).
	\end{align*}
	For the second assertion, consider the sum $Z^k_\beta+ \delta^*(H^{k+1}_{\beta+1}(T^*M))$.
	Due to the first assertion, we have  
	\begin{align*}
		Z^k_{\beta}\cap \delta^*(H^{k+1}_{\beta+1}(T^*M))	=0,
	\end{align*}
	so the sum is direct. It remains to show that the sum is equal to $B^{-1}(\Delta_1(H^{k+1}_{\beta+1}(T^*M)))$. For an arbitrary $h\in B^{-1}(\Delta_1(H^{k+1}_{\beta+1}(T^*M)))$, we choose a form $\omega_0\in H^{k+1}_{\beta+1}(T^*M)$ such that $Bh=\Delta_1\omega_0$ so that $h-\delta^*\omega_0\in 	\kernel_{H^k_{\beta}}(B)$. Now by the first assertion, we can write $h-\delta^*\omega_0=\delta^*\omega_1+h_1$, with some $h_1\in Z^k_{\beta}$ and $h=\delta^*(\omega_0+\omega_1)+h_1$ is the desired decomposition.
\end{proof}

The following proposition asserts that the Bianchi gauge is a very reasonable one for any asymptotically conical Ricci-flat manifold $(M,g)$: Any metric $\tilde{g}$ sufficiently close to $g$ can be brought by a diffeomorphism into Bianchi gauge, possibly up to an open subset we are free to choose.
\begin{thm}\label{slice}
	Let $(M^n,g)$ be an AC Ricci-flat manifold and let $k>n/2+1$ and $\beta<0$ be such that $\beta+1$ is a nonexceptional value for $\Delta_1$ and $\beta$ is a nonexceptional value for $B$.  Pick an open and precompact subset $U\subset M$ and a complement $Z^k_{\beta}$ as \eqref{eq:complement-of-delta-star-kernel}.
	Set
	\begin{align*}
		\mathcal{H}_{\beta}^k:=\SetDefine{g+h}{h\in Z^k_{\beta}}.
	\end{align*} 
	Then there exists an $H^k_{\beta}$-neighborhood $\mathcal{U}^k_{\beta}$ of $g$ in the space of metrics such that for any $\tilde{g}\in \mathcal{U}_{\beta}^k$, there exists a diffeomorphism $\varphi$ which is $H^{k+1}_{\beta+1}$-close to the identity and a tensor $\tilde{h}$ with $\supp(\tilde{h})\subset\ol{U}$ such that $\varphi^*(\tilde{g}+\tilde{h})\in \mathcal{H}_{\beta}^k$.
\end{thm}
\begin{rem}
	If $\beta<-1$, we get	$Z^k_{\beta}=\kernel_{H^k_{\beta}}(B)$ and hence $\mathcal{H}_{\beta}^k=\mathcal{G}_{\beta}^k$ because $\Delta_1$ is injective on $H^{k+1}_{\beta+1}(T^*M)$, see Lemma~\ref{prop:surj_Laplace}. On the other hand, if $\beta>1-n$, the assertion holds without adding the tensor $k$, because $\Delta_1$ is then surjective on $H^{k+1}_{\beta+1}(T^*M)$, see again Lemma~\ref{prop:surj_Laplace}.
\end{rem}
\begin{rem}\label{rem:gauge_mutiple_ends}
	A completely analogous statement holds in the case of one conically singular end, with the only difference that $\beta$ is chosen positive. In case of multiple ends, we also have an analogous assertion. In this case one would work with function spaces with multiple weights which were briefly introduced in Remark~\ref{rem:multiple ends}. One would then choose a tuple $\beta=(\zeta,\sigma)$ with $\zeta<0$ and $\sigma>0$.
\end{rem}
\begin{proof}
	Clearly, we have 
	\begin{align*}
		T_{g}\mathcal{H}_{\beta}^k=Z^k_{\beta}.
	\end{align*}
	Choose tensors $h_i$, $i\in\SetDefine{1,\ldots,N}{}$ with support in $\ol{U}$ as in Lemma~\ref{lem:compact_perturbation2} and let $\wt{V}=\mathrm{span}_{\R}(h_i)_{1\leq i\leq N}$.	
	Then, Lemma~\ref{lem:complement_gauge} implies that 
	\begin{align}\label{eq:gauge_decomp2}
		H^k_{\beta}(S^2M)=	 Z^k_{\beta}\oplus \delta^*(H_{\beta+1}^{k+1}(T^*M))\oplus \wt{V}=
		T_{g}\mathcal{H}_{\beta}^k\oplus \mathcal{L}\circ \sharp(H_{\beta+1}^{k+1}(T^*M))	\oplus \wt{V}.
	\end{align}
	Here, we used that $2\delta^*=\mathcal{L}\circ \sharp$, where $\mathcal{L}:X\to \mathcal{L}_Xg$ is the Lie derivative and $\sharp:\omega\mapsto\omega^{\sharp}$ is the sharp operator with respect to $g$.	
	Note that because $\beta+1<1$, the vector fields in $H^{k+1}_{\beta+1}(TM)$ are all complete, because they grow slower than linearly. Therefore, we have a well-defined map
	\begin{align*}
		H^{k+1}_{\beta+1}(TM)\ni X\mapsto \psi_X\in H^{k}_{\beta}(\mathrm{Diff}(M)),	
	\end{align*}
	where $\psi_X$ is the flow of $X$, evaluated at time $t=1$.
	Now we consider the smooth map
	\begin{align*}
		\Psi:  \mathcal{H}_{\beta}^k\times H^{k+1}_{\beta+1}(TM)\times \wt{V}\to H_{\beta}^k(S^2_+M),\qquad
		(g,X,\tilde{h})\mapsto (\psi_X)^*g+\tilde{h}.
	\end{align*}
	Its differential at $(g,0,0)$ corresponds to the decomposition \eqref{eq:gauge_decomp2}. Therefore, $\Psi$ is a local diffeomorphism from a neighborhood of $(g,0,0)$ onto a neighborhood of $g$ by the implicit function theorem. This proves the Proposition.
\end{proof}
\begin{rem}
	Note that the essential reason for the assumption $\beta<0$ is to guarantee completeness for the vector fields in $H^{k+1}_{\beta+1}$. In the conically singular case, completeness is guaranteed by assuming $\beta>0$ as this implies that the vector fields decay at the singularity like $\mathcal{O}(r^{1+\epsilon})$.
\end{rem}

\subsection{Proof of the main results}
This section is devoted to the proof of Theorem~\ref{mainthm:conifold_rate} and essentially builds up on the slice theorem \ref{slice}. We prove only for the case of one AC end. The case of one conically singular end is completely analogous. The proof for multiple ends of both types is also analogous and based on a slice theorem for multiple ends, see Remark~\ref{rem:gauge_mutiple_ends}. The details are left to the reader. Recall the notation $\ol{M}_{<R}$ and $\ol{M}_{>R}$ from \eqref{eq:sublevel-set-of-the-radius-function}.
\begin{thm}\label{thm:AC_chart}
	Let $(M,g)$ be an asymptotically conical Ricci-flat manifold. Then there exist compact set $K\subset M$ and an asymptotic chart $\varphi:M\setminus K\to \ol{M}_{>R}$ such that
	\begin{itemize}
		\item[(i)] if $(\ol{M},\ol{g})$ is not resonance-dominated, we have $\varphi_*g-\ol{g}\in \mathcal{O}_{\infty}(r^{-\xi_-})$ as $r\to\infty$ where $\xi_-$ has been defined in \eqref{eq:indicial_roots},
		\item[(ii)] if $(\ol{M},\ol{g})$ is resonance-dominated, we have $\varphi_*g-\ol{g}\in \mathcal{O}_{\infty}(r^{-\frac{n-2}{2}}\log(r))$ as $r\to\infty$.	
	\end{itemize}
\end{thm}
\begin{figure}[tbh!]
	\centering
	\includegraphics{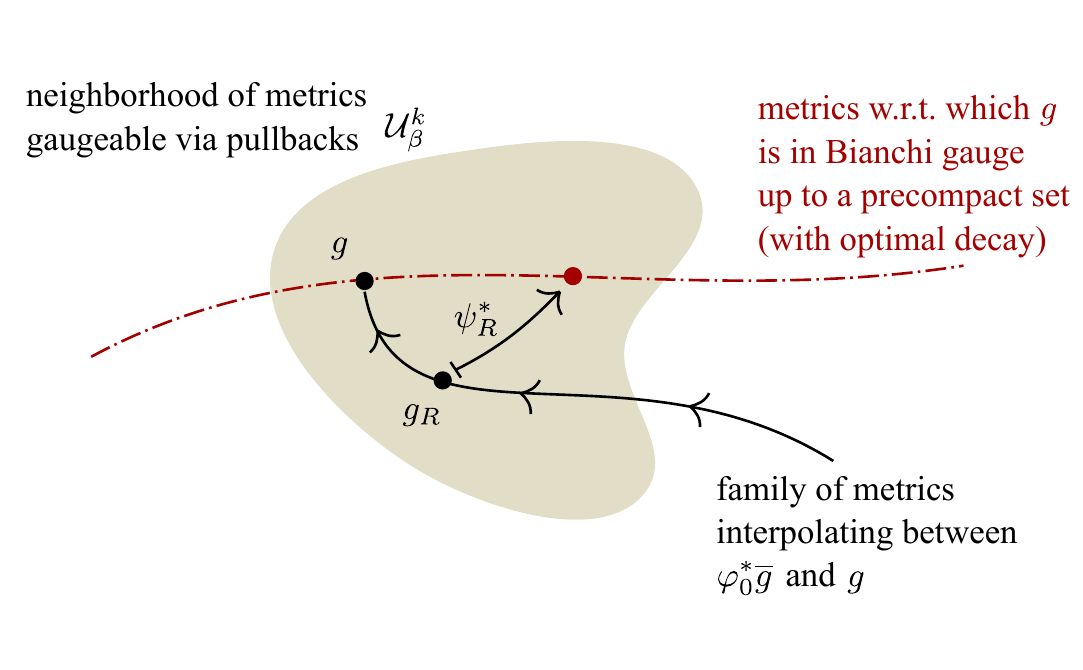}
	\caption[Construction of the new asymptotic chart]{Construction of the new asymptotic chart in the proof of Theorem~\ref{thm:AC_chart}. The picture takes place in the set of metrics on $M$. The dash-dotted line represents metrics that are (up to a precompact set) in Bianchi gauge w.r.t. $g$. The shaded region represents the neighborhood $\mathcal U^k_\beta$ of metrics gaugeable by pullbacks, cf. Theorem~\ref{slice}. The family of metrics $g_R$ converges to $g$, therefore it will eventually enter the neighborhood $\mathcal U^k_\beta$.}
	\label{fig:new-asymptotic-chart}
\end{figure}
\begin{proof}
	Because $(M,g)$ is asymptotically conical, we find a compact set $K_0\subset M$ and a diffeomorphism $\varphi_0:M\setminus K_0\to \ol{M}_{>R}$  such that $(\varphi_0)^*g-\ol{g}\in \mathcal{O}_{\infty}(r^{-\tau})$ for some $\tau>0$ as $r\to\infty$. Let $f_R:\R_{\geq 0}\to [0,1]$ be a smooth cutoff function such that
	\begin{align*}
		f_R|_{[0,R]}\equiv 1,\qquad f_R|_{[2R,\infty)}\equiv 0,
		\qquad |f^{(k)}_R|\leq C_k\cdot R^{-k}\text{ for all }k\in\N
	\end{align*}
	for some constants $C_i\in\R$.
	Choose a radius function $\rho$ on $M$ such that $(\varphi_0)_*\rho=r$
	and  let $F_R=f_R\circ \rho$. 
	Let $R_0:=\max_K \rho$, and for $R\geq R_0$, define a metric $g_R$ on $M$ by
	\begin{align*}
		g_R=F_R\cdot g+(1-F_R)(\varphi_0)^*\ol{g}.
	\end{align*}
	Inside $M_{<R}$, the metric $g_R$ agrees with $g$ and on $M_{>2R}$, the metric $g_R$ agrees with $(\varphi_0)^*\ol{g}$. Moreover, by the choice of $F_R$, we have
	$g_R\to g$ in $ C^{k}_{-\tau}(S^2M)$ for any $k\in \N$ as $R\to \infty$. Thus for any $\beta>-\tau$,
	we also get $g_R\to g$ in $H^k_{\beta}(S^2M)$
	for any $k\in \N$.
	Choose $k>n/2+2$ pick an open set $U\subset K_0\subset M$.  
	Then by Theorem~\ref{slice}, there exists for a sufficiently large value for $R$ (which we denote again by $R$), a diffeomorphism $\psi_{R}$ 
	and a tensor $h_R$  with compact support in $\ol{U}\subset K_0$ such that
	such that
	\begin{align}\label{eq:gauging-up-to-a-compact-set}
		V(g,(\psi_R)^*(g_{R}+h_R))=0.	
	\end{align}
	Consider now the compact set $K:=(\varphi_{R})^{-1}(K_0)$ and the diffeomorphism
	\begin{align*}
		\varphi:=\varphi_0\circ\psi_{R}: M\setminus K\to\ol{M}_{>R}.	
	\end{align*}
	By construction, the function $F_{R}$ vanishes on the set $M_{>2R} \subset M$.
	Furthermore, $h_R$ vanishes on the domain of $\varphi_0$.
	Therefore we have on the set $\ol{M}_{>2R}$ that $\varphi_*(\psi_R)^*g_{R}=\varphi_*(\psi_R)^*(\varphi_0)^*\ol{g}=\ol{g}$ and consequently by \eqref{eq:gauging-up-to-a-compact-set}
	\begin{align*}
		-2\ric_{\varphi_*g}+	\mathcal{L}_{V(\varphi_*g,\ol{g})}(\varphi_*g)=	\mathcal{L}_{V(\varphi_*g,\varphi_*(\psi_R)^*g_{R_1})}(\varphi_*g)=0.	
	\end{align*}
	Because $k$ was chosen to satisfy $k>n/2+2$, we have $\varphi_*g-	\ol{g}=\mathcal{O}_{2}(r^{-\beta/2}) $. 
	From Theorem~\ref{thm:Ricci_flat_bianchi_gauge}, we obtain that
	\begin{align*}
		\varphi_*g-	\ol{g}=\mathcal{O}_{\infty}(r^{-\xi^B_-}). 	
	\end{align*}
	If $\xi^B_-=\xi_-$, we are done.
	If $\xi^B_-<\xi_-$, we know 
	that the leading term of the expansion at infinity is a Lie derivative. In other words, we can decompose
	\begin{align}\label{eq:lie_decomposition}
		\varphi_*g-	\ol{g}=\varphi_*(g-(\psi_R)^*g_{R})=\ol{\delta}^*\ol{\omega}+\hat{h},\qquad  \hat{h}=\mathcal{O}_{\infty}(r^{-\xi^B_--\epsilon}),	
	\end{align}
	for some $\epsilon>0$ and $\ol{\omega}$. In the following, we are going to show that the leading gauge term $\ol{\delta}^*\ol{\omega}$ indeed vanishes.
	Consider the difference
	\begin{align*}
		h_0:=(\psi_R)^*(g_{R}+h_R)-g\in H^k_{\beta}(S^2M),\qquad\text{ for }\qquad k>\frac{n}{2}+2,\qquad \beta>-\xi^B_-.
	\end{align*}
	Consider the 1-form $\ol{\omega}$ in \eqref{eq:lie_decomposition} and extend the form $\varphi^*\ol{\omega}$ on $M\setminus K$ to a smooth 1-form $\omega_1$ on all of $M$. Then, we have that 
	\begin{align*}
		h_0=\delta^*\omega_1+h_1,\qquad h_1=\mathcal{O}_{\infty}(r^{-\xi^B_--\epsilon}) .
	\end{align*}
	for some $\epsilon>0$.
	Recall that from the proof of Proposition~\ref{slice}, we have $h_0\in Z^k_{\beta}$, where $Z^k_{\beta}$ is a space such that
	\begin{align*}
		\kernel_{H^k_{\beta}}(B)=Z^k_{\beta}\oplus \delta^*(\kernel_{H^{k+1}_{\beta+1}}(\Delta_1)).
	\end{align*}
	Let us proceed with the tensor $h_1$. At first we know
	\begin{align*}
		h_1\in 	 H^k_{\beta'}(S^2M),\qquad\text{ for }\qquad  k>\frac{n}{2}+2,\qquad \beta'>-\xi^B_--\epsilon.
	\end{align*}
	Now choose $\beta'\in (\xi^B_--\epsilon,\xi^B_-)$ and choose a subspace $Z^k_{\beta'}$ with 
	\begin{align*}
		\kernel_{H^k_{\beta'}}(B)=Z^k_{\beta'}\oplus \delta^*(\kernel_{H^{k+1}_{\beta'+1}}(\Delta_1))
	\end{align*}
	which additionally satisfies
	\begin{align*}
		Z^k_{\beta'}\subset Z^k_{\beta}.
	\end{align*}
	Due to \eqref{eq:gauge_decomp2}, we can write 
	\begin{align*}
		h_1=\delta^*\omega_2+h_2+k_2
	\end{align*}
	with $\omega_2\in H^{k+1}_{\beta'+1}$ and $k_2$ is a tensor with support in a small region. Rearranging yields
	\begin{align*}
		h_0=\delta^*(\omega_1+\omega_2)+h_2+k_2\subset \delta^*(H^k_{\beta}(T^*M))\oplus 	Z^k_{\beta'} \oplus \tilde{V}
		\subset \delta^*(H^k_{\beta}(T^*M))\oplus 	Z^k_{\beta} \oplus \tilde{V}	,	 
	\end{align*}
	where we also have taken into account that the sums on the right-hand side are all direct by \eqref{eq:gauge_decomp2}.
	But because $h_0\in Z^k_{\beta}$, this actually implies that $\delta^*(\omega_1+\omega_2)=0$.
	Thus,
	\begin{align*}
		h_0=h_2+k_2\in 	Z^k_{\beta'} \oplus \tilde{V}\subset H^k_{\beta'}(S^2M)
	\end{align*}
	therefore
	\begin{align*}
		h_0=\O_{2}(r^{-\xi_-^B-\epsilon'})	
	\end{align*}
	for some $\epsilon'>0$. Pulling back to the cone $\ol{M}$, this implies (by using elliptic regularity) that 
	\begin{align}%
		\varphi_*g-	\ol{g}=\varphi_*(g-g_R)=\mathcal{O}_{\infty}(r^{-\xi^B_--\epsilon'}).
	\end{align}
	By successively improving the decay rate as in the proof of Theorem~\ref{thm:AC_chart} and repeating the above procedure a finite number of times, we obtain
	\begin{align}%
		\varphi_*g-	\ol{g}=\varphi_*(g-g_R)=\mathcal{O}_{\infty}(r^{-\xi_-}),
	\end{align}
	as desired.
\end{proof}
With a combination of the analysis done in Theorem~\ref{thm:Ricci_flat_bianchi_gauge} and Theorem~\ref{thm:AC_chart}, one also obtains the following result:
\begin{thm}\label{thm:common_AC_chart}
	Let $(M,g)$ be a Ricci-flat AC manifold and $\tilde{g}$ be another Ricci-flat metric which is in Bianchi gauge with respect to $g$. Assume that $g-\tilde{g}=\mathcal{O}(r^{-\alpha})$ for some $\alpha>0$, as $r\to\infty$. Then,
	\begin{itemize}
		\item[(i)] if $(M,g)$ is not resonance-dominated, we have $g-\tilde{g}\in \mathcal{O}_{\infty}(r^{-\xi_-})$ as $r\to\infty$ where $\xi_-$ has been defined in \eqref{eq:indicial_roots},
		\item[(ii)] if  $(M,g)$ is resonance-dominated, we have $g-\tilde{g}\in \mathcal{O}_{\infty}(r^{-\frac{n-2}{2}}\log(r))$ as $r\to\infty$.	
	\end{itemize}
	In view of Theorem~\ref{thm:AC_chart} this means the following: For both metrics $g,\tilde{g}$, we can pick one common asymptotic chart for which we have the optimal decay rate from Theorem~\ref{thm:AC_chart}. 
\end{thm}

We conclude this paper with computing the order of ALE manifolds, by direct application of Theorem~\ref{thm:AC_chart}.
\begin{proof}[Proof of Theorems~\ref{mainthm:orbifold_Rate} and Theorem~\ref{mainthm:ALE_Rate}]
	We first have to consider the eigenvalue data $\lambda_i$, $\mu_i$ and $\kappa_i$ for quotients $S^{n-1}/\Gamma$. At first, we have
	\begin{align}\label{eq:Laplace_spectrum_sphere}
		\spectrum(\Delta,S^{n-1}/\Gamma)\subset\spectrum(\Delta,S^{n-1})=\SetDefine{\lambda_i=i(i+n-2)}{i\in\N_0},
	\end{align}
	see e.g.\ \cite{BGM71},
	and by the equality case in the Lichnerowicz--Obata eigenvalue inequality, we have
	\begin{align*}
		\spectrum(\Delta,S^{n-1}/\Gamma)\subset\SetDefine{\lambda_i}{i\in\N_0\setminus\SetDefine{1}{}},
	\end{align*}
	whenever $\Gamma\neq\SetDefine{1}{}$. 
	In \cite[Theorem~3.2]{Bou99}, Boucetta computed the spectrum of the Lichnerowicz Laplacian on $S^n$. We conclude
	\begin{align}\label{eq:Einstein_spectrum_sphere}
		\spectrum(\Delta_E|_{TT(S^{n-1}/\Gamma)})\subset\spectrum(\Delta_E|_{TT(S^{n-1})})=\SetDefine{\kappa_i=(i+1)(i+n-1)}{i\in\N}.	
	\end{align}
	Note that the differences between \eqref{eq:Einstein_spectrum_sphere} and the values in \cite[Theorem~3.2]{Bou99} come from shifting the dimension and the eigenvalue index by one, and by switching from $\Delta_L$ to $\Delta_E$. 
	Recalling the notation from Subsection~\ref{subsec:order} (see \eqref{eq:rates}) we get in this case that
	\begin{alignat*}{99}
		\xi_+&:=\min E_+=\min\SetDefine{\xi_+(\kappa_i),\xi_+(\lambda_i)}{i\in\N}&&\geq2\\	
		\xi_-&:=\min E_-=\min\SetDefine{-\xi_-(\kappa_i),-\xi_-(\lambda_i)}{i\in\N}&&\geq n.
	\end{alignat*}
	The results now follow from Theorem~\ref{mainthm:conifold_rate}.
\end{proof}

\end{document}